%% file: ricks-counting.tex
\documentclass{amsart}

\input{artcfg}

\renewcommand{\Bu}{\ensuremath \Bset^{\rm unclipped}}

\newcommand{\fn}{\psi}
\newcommand{\varalpha}{\alpha}

\title[Counting periodic geodesics in a compact locally CAT(0) space]{Counting closed geodesics in a compact rank one locally CAT(0) space}

\author{Russell Ricks}
\address{Binghamton University, Binghamton, New York, USA}
\email{ricks@math.binghamton.edu}

\date{\today}

\begin{document}

\begin{abstract}
Let $X$ be a compact, geodesically complete, locally CAT(0) space such that the universal cover admits a rank one axis.
Assume $X$ is not homothetic to a metric graph with integer edge lengths.
Let $P_t$ be the number of parallel classes of oriented closed geodesics of length $\le t$; then $\lim\limits_{t \to \infty} P_t / \frac{e^{ht}}{ht} = 1$, where $h$ is the entropy of the geodesic flow on the space $SX$ of parametrized unit-speed geodesics in $X$.
\end{abstract}

\maketitle

\section{Introduction}

Given a locally geodesic space, it is natural to consider the number $P_t$ of closed geodesics of length at most $t > 0$.
In general, $P_t$ may be infinite for all $t$ above a certain threshold $T \ge 0$, but under certain geometric conditions one finds it is finite for all $t$ and can obtain asymptotic information about the growth rate of $P_t$.

The classic example of this situation is a theorem of Margulis \cite{margulis}:
If $M$ is a closed, negatively-curved Riemannian manifold, then $\lim\limits_{t \to \infty} \slashfrac{P_t}{\frac{e^{ht}}{ht}} = 1$, where $h$ is the entropy of the geodesic flow on the unit tangent bundle $SM$.
Margulis also proved that the number $Q_t$ of geodesic arcs of length $\le t$ starting at $x \in M$ and ending at $y \in M$, satisfies $\lim\limits_{t \to \infty} \slashfrac{Q_t}{e^{ht}} = C$, where $C$ depends only on $x,y$.

In nonpositive curvature (instead of strictly negative curvature), there are often parallel geodesics, which can make the number $P_t$ as defined above infinite for large $t$.
However, if one refines the definition of $P_t$ to be the number of \emph{parallel classes} of closed geodesics of length $\le t$, it becomes meaningful again in this case, while staying the same in the case of negative curvature.
Knieper \cite{knieper98} proved that when $M$ is a closed, rank one nonpositively-curved Riemannian manifold, there exists $C > 0$ such that
$\frac{1}{C} \le \liminf \slashfrac{P_t}{\frac{e^{ht}}{ht}}$ and $\limsup \slashfrac{P_t}{e^{ht}} \le C$.
Knieper later improved his bounds \cite{knieper01} to
$\frac{1}{C} \le \liminf \slashfrac{P_t}{\frac{e^{ht}}{ht}} \le \limsup \slashfrac{P_t}{\frac{e^{ht}}{ht}} \le C$.
(This type of inequality occurs frequently enough in this paper that we will use the notation $\fakelim$ when the inequality holds for both $\liminf$ and $\limsup$.
In this notation, the last inequalities become
$\frac{1}{C} \le \fakelim \slashfrac{P_t}{\frac{e^{ht}}{ht}} \le C$.)
Knieper's original bounds were recently proved by different means by Burns, Climenhaga, Fisher, and Thompson \cite{bcft18}.
A recent preprint \cite{liu-wang-wu} generalizes this beyond nonpositive curvature to the case of closed Riemannian manifolds without focal points.

Another way to generalize the setting of Margulis' theorem is to allow the spaces to admit singularities.
In fact, locally $\CAT(-1)$ spaces are a generalization of negatively-curved manifolds which allow branching and other singularities.
They are locally geodesic spaces in which all sufficiently small geodesic triangles are ``thinner'' than their respective comparison triangles in the hyperbolic plane $\H^2$.
Roblin proved \cite{roblin} that if the Bowen--Margulis measure of a proper, locally $\CAT(-1)$ space is finite, then
$\lim\limits_{t \to \infty} \slashfrac{Q_t}{e^{ht}} = C$, where $C$ depends only on $x,y$.
A recent preprint by Link \cite{link} generalizes this statement from $\CAT(-1)$ to rank one $\CAT(0)$.
Locally $\CAT(0)$ spaces generalize nonpositively-curved manifolds by allowing singularities; the definition uses comparison triangles in the Euclidean plane $\R^2$ instead of $\H^2$.
Roblin also proved \cite{roblin} that if the Bowen--Margulis measure of a proper, locally $\CAT(-1)$ space $X$ is finite and mixing, and $X$ is geometrically finite, then
$\lim\limits_{t \to \infty} \slashfrac{P_t}{\frac{e^{ht}}{ht}} = 1$.
\footnote{Technically, Roblin and Link do not address the question of entropy.
The constant $h$ used here is actually $\delta_\G$ the critical exponent of the Poincar\'e series for $\G$ (see \secref{measures}).
At least in the case where $\G$ acts cocompactly, $\delta_\G$ equals the topological entropy $h$.}

In this paper, we focus on the case of proper, rank one, locally $\CAT(0)$ spaces.
We assume throughout the paper (with the exception of \secref{pi-convergence section}) that $\G$ is a group acting freely, properly discontinuously, non-elementarily, and by isometries on a proper, geodesically complete $\CAT(0)$ space $X$ with rank one axis.
We also assume the geodesic flow is mixing and the Bowen-Margulis measure (constructed in \cite{ricks-mixing}) is finite and mixing under the geodesic flow.
When $\G$ acts cocompactly, it is well-known to also act non-elementarily unless $X$ is isometric to the real line; in \cite{ricks-mixing} it was shown that cocompactness also implies the Bowen-Margulis measure is always finite and mixing unless $X$ is homothetic to a tree with integer edge lengths.
We prove the following.

\begin{maintheorem}				\label{main}
Let $\G$ be a group acting freely, geometrically (that is, properly discontinuously, cocompactly, and by isometries) on a proper, geodesically complete $\CAT(0)$ space $X$ with rank one axis.
Assume $X$ is not homothetic to a tree with integer edge lengths.
Let $P_t$ be the number of parallel classes of oriented closed geodesics of length $\le t$ in $\modG{X}$; then $\lim\limits_{t \to \infty} P_t / \frac{e^{ht}}{ht} = 1$, where $h$ is the entropy of the geodesic flow on the space $SX$ of parametrized unit-speed geodesics in $X$.
\end{maintheorem}

We remark that if $X$ is homothetic to a tree with integer edge lengths, then the limit of $P_t / \frac{e^{ht}}{ht}$ does not exist.
Also, the closed geodesics which bound a half flat in the universal cover (called the singular geodesics) grow at a strictly smaller exponential rate.

We note that a recent preprint \cite{gekhtman-yang} generalizes Knieper's bounds
$\frac{1}{C} \le \fakelim \slashfrac{P_t}{\frac{e^{ht}}{ht}} \le C$
to the proper, rank one, locally $\CAT(0)$ case.
We prove the exact limit.
We also note that an unpublished paper from 2007 by Roland Gunesch \cite{gunesch} claims our result for compact, rank one, nonpositely-curved manifolds.
Indeed, many of the ideas in Gunesch's work are good and inspired the current paper.

We proceed as follows in the paper.
First, after establishing notation and standard facts about rank one $\CAT(0)$ spaces, we use Papasoglu and Swenson's $\pi$-convergence theorem to prove a statement about local uniform expansion along unstable horospheres.
Next, we construct product boxes (which behave better than standard flow boxes for measuring lengths of intersection for orbits), and use mixing to prove a result about the total measure of intersections under the flow for these product boxes.
We use this to count the number of intersections coming from periodic orbits.
Then we construct measures equally-weighted along periodic orbits. We adapt Knieper's proof of an equidistribution result to prove \thmref{main}.

\section{Preliminaries}

A \defn{geodesic} in a metric space $X$ is an isometric embedding of the real line $\R$ into $X$.
A \defn{geodesic segment} is an isometric embedding of a compact interval, and a \defn{geodesic ray} is an isometric embedding of $[0, \infty)$.

A metric space $X$ is called \defn{uniquely geodesic} if for every pair of distinct $x,y \in X$ there is a unique geodesic segment $u \colon [a,b] \to X$ such that $u(a) = x$ and $u(b) = y$.
The space $X$ is \defn{geodesically complete} (or, $X$ has the \defn{geodesic extension property}) if every geodesic segment in $X$ extends to a full geodesic in $X$.

A \defn{CAT(0) space} is a uniquely geodesic space such that for every triple of distinct points $x,y,z \in X$, the geodesic triangle is no fatter than the corresponding comparison triangle in Euclidean $\R^2$ (the triangle with the same edge lengths).
A detailed account of $\CAT(0)$ spaces is found in \cite{ballmann} or \cite{bridson}.

Every complete $\CAT(0)$ space $X$ has an \defn{ideal boundary}, written $\bd X$, obtained by taking equivalence classes of asymptotic geodesic rays.
The compact-open topology on the set of rays induces a topology on $\bd X$, called the \defn{cone} or \defn{visual} topology.
If $X$ is proper (meaning all closed balls are compact), then both $\bd X$ and $\overline{X} = X \cup \bd X$ are compact metrizable spaces.

STANDING HYPOTHESIS:
From now on, $X$ will always be a proper, geodesically complete $\CAT(0)$ space.

Denote by $SX$ the space of all geodesics $\R \to X$, where $SX$ is endowed with the compact-open topology.
Then $SX$ is naturally a proper metric space, and there is a canonical footpoint projection map $\pi \colon SX \to X$ given by $\pi(v) = v(0)$; this map is proper.
There is also a canonical endpoint projection map $\emap \colon SX \to \dbX$ defined by $\emap(v) = (v^-, v^+) := (\lim_{t \to -\infty} v(t), \lim_{t \to +\infty} v(t))$.
And $w \in SX$ is parallel to $v \in SX$ if and only if $\emap(w) = \emap(v)$.

The \defn{geodesic flow} $g^t$ on $SX$ is defined by the formula $(g^t v)(r) = v(t + r)$.

A geodesic $v$ in $X$ is called \defn{higher rank} if it can be extended to an isometric embedding of the half-flat $\R \times [0, \infty) \subseteq \R^2$ into $X$.
A geodesic which is not higher rank is called \defn{rank one}.
Let $\Reg \subseteq SX$ denote the set of rank one geodesics.
The following lemma describes an important aspect of the geometry of rank one geodesics in a $\CAT(0)$ space.

\begin{lemma}[Lemma III.3.1 in \cite{ballmann}]\label{Ballmann's lemma}
Let $w \colon \R \to X$ be a geodesic which does not bound a flat strip of width $R > 0$.  Then there are neighborhoods $U$ and $V$ in $\bar X$ of the endpoints of $w$ such that for any $\xi \in U$ and $\eta \in V$, there is a geodesic joining $\xi$ to $\eta$.  For any such geodesic $v$, we have $d(v, w(0)) < R$; in particular, $v$ does not bound a flat strip of width $2R$.
\end{lemma}

Define the \defn{cross section} of $v \in SX$ to be $CS(v) = \pi_p^{-1} \set{\pi_p(v)}$, and the \defn{width} of a geodesic $v \in SX$ to be $\width(v) = \diam CS(v)$.
The width of $v$ is in fact the maximum width of a flat strip $\R \times [0, R]$ in $X$ parallel to $v$.

Now let $\G$ be a group acting properly discontinuously, by isometries on $X$.
The $\G$-action on $X$ naturally induces an action by homeomorphisms on $\overline{X}$ (and therefore on $\bd X$).
The \defn{limit set} of $\G$ is $\Limitset = \cl{\G x} \cap \bd X$, for some $x \in X$.
The limit set is closed and invariant, and it does not depend on choice of $x$.
The action is called \defn{elementary} if either $\Limitset$ contains at most two points, or $\Gamma$ fixes a point in $\bd X$.

The $\G$-action on $X$ also induces a properly discontinuous, isometric action on $SX$.
Denote by $g_\G^t$ the induced flow on the quotient $\modG{SX}$, and let $\pr \colon SX \to \modG{SX}$ be the canonical projection map.

A geodesic $v \in SX$ is \defn{axis} of an isometry $\g \in \Isom X$ if $\g$ translates along $v$, i.e., $\g v = g^t v$ for some $t > 0$.
If some rank one geodesic $v \in \Reg$ is an axis for $\g \in \Isom X$, we call $\g$ \defn{rank one}.
We call the $\G$-action \defn{rank one} if some $\g \in \G$ is rank one.

STANDING HYPOTHESIS:
$\Gamma$ is a group acting properly discontinuously, by isometries on $X$.
Except in \secref{pi-convergence section}, we further assume the action is rank one, non-elementary, and free (that is, no nontrivial $\g \in \G$ fixes a point of $x \in X$).

\section{Locally Uniform Expansion along Unstable Horospheres}
\label{pi-convergence section}

There is a topology on $\bd X$, finer than the visual topology, that comes from the \defn{Tits metric} $\dT$ on $\bd X$.
The Tits metric is complete $\CAT(1)$, and measures the asymptotic angle between geodesic rays in $X$.
In fact, a geodesic $v \in SX$ is rank one if and only if $\dT(v^-,v^+) > \pi$.
Write $\BT (\xi, r)$ for the open Tits ball of $\dT$-radius $r$ about $\xi$ in $\bd X$ and $\cl{\BT} (\xi, r)$ for the closed ball.

Papasoglu and Swenson's $\pi$-convergence:

\begin{theorem}[Lemma 18 of \cite{ps}]		\label{pi-convergence}
Let $X$ be a proper $\CAT(0)$ space and $G$ a group acting by isometries on $X$.
Let $x \in X$, $\theta \in [0,\pi]$, and $(g_i) \subset G$ such that $g_i (x) \to p \in \bd X$ and $g_i^{-1} (x) \to n \in \bd X$.
For any compact set $K \subset \bd X \setminus \cl{\BT} (n, \theta)$,
\(g_i (K) \to \cl{\BT} (p, \pi - \theta),\)
(in the sense that for any open $U \supset \cl{\BT} (p, \pi - \theta),
g_i (K) \subset U$ for all $i$ sufficiently large). \end{theorem}

From \thmref{pi-convergence} we prove that the geodesic flow expands unstable horospheres locally uniformly (\thmref {regular uniform expansion}).

\begin{lemma}					\label{eval}
The evaluation map $\mathrm{ev} \colon SX \times (-\infty, \infty) \to X$ given by $\mathrm{ev}(v,t) = v(t)$ extends continuously to a map $SX \times [-\infty, \infty] \to \cl X$.
\end{lemma}

\begin{lemma}					\label{curious lemma}
Let $\Gamma$ be a group acting properly isometrically on a proper $\CAT(0)$ space $X$.
Let $\Aset \subset SX$ be compact.
Let $\Aset^- = \setp{v^-}{v \in \Aset}$ and $\Aset^+ = \setp{v^+}{v \in \Aset}$.
Let $(\g_i)$ be a sequence in $\G$ such that $\g_i x \to \xi \in \bd X$ for some (hence any) $x \in X$ and $\Aset \cap \g_i g^{-t_i} \Aset \neq \varnothing$ for some sequence $(t_i)$ in $[0, \infty)$.
Then $\xi \in \Aset^+$.
Let $K \subset \bd X$ be compact such that $\dT(\Aset^-, K) > \pi - c$ for some $c \in [0,\pi]$.
If $U \subseteq \bd X$ is an open set such that $\cl{\BT} (\xi, c) \subseteq U$, then $\g_i (K) \subseteq U$ for all $i$ sufficiently large.
\end{lemma}

\begin{proof}
First observe that the sets
$\pi(g^{[0,\infty]} \Aset) =
\Aset^+ \cup \setp{v(t)}{v \in \Aset \text{ and } t \ge 0}$
and $\pi(g^{[-\infty,0]} \Aset) =
\Aset^- \cup \setp{v(t)}{v \in \Aset \text{ and } t \le 0}$
are closed in $\cl X$ because $\Aset$ is compact.

For each $i \in \N$, let $v_i \in \Aset \cap \g_i g^{-t_i} \Aset$.
Passing to a subsequence if necessary, we may assume the sequence $(v_i)$ converges to some $v_0 \in \Aset$, and $(\g_i^{-1} g^{t_i} v_i)$ converges to some $w_0 \in \Aset$.
Let $x_0 = v_0(0)$ and $y_0 = w_0(0)$.
Recall that $\g_i y_0 \to \xi \in \bd X$.
We may assume the sequence $(\g_i^{-1} x_0)$ converges to some $\eta \in \bd X$.

We know $d(\g_i w_0, g^{t_i} v_i) \to 0$, so $d(\g_i y_0, v_i(t_i)) \to 0$.
Since $\pi(g^{[0,\infty]} \Aset)$ is closed, we may conclude $\xi = \lim v_i(t_i) \in \Aset^+$.
Now for each $i \in \N$ let $w_i = \g_i^{-1} g^{t_i} v_i$.
Then $d(\g_i^{-1} v_0, g^{-t_i} w_i)
= d(\g_i^{-1} v_0, \g_i^{-1} v_i)
\to 0$,
and so $d(\g_i^{-1} x_0, w_i(-t_i)) \to 0$.
Since each $w_i \in \Aset$ and $\pi(g^{[-\infty,0]} \Aset)$ is closed, we see that $\eta = \lim w_i(-t_i) \in \Aset^-$.

Thus $\g_i x_0 \to \xi \in \Aset^+$ and $\g_i^{-1} x_0 \to \eta \in \Aset^-$.
Apply \thmref{pi-convergence}.
\end{proof}

\begin{theorem}					\label{singular uniform expansion 2}
Let $X$ be a proper $\CAT(0)$ space and $\Gamma$ a group acting properly isometrically on $X$.
Let $\Aset \subset SX$ be compact.
Let $\Aset^- = \setp{v^-}{v \in \Aset}$ and $\Aset^+ = \setp{v^+}{v \in \Aset}$.
Let $c \in [0, \pi]$ and let $\set{U_\lambda}$ be an open cover of $\Aset^+$ such that for every $\xi \in \Aset^+$, there is some $\lambda$ such that $\cl{\BT} (\xi, c) \subseteq U_\lambda$.
For any compact set $K \subset \bd X$ such that $\dT (\Aset^-, K) > \pi - c$, there is some $t_0 \ge 0$ such that for all $t \ge t_0$ and $\gamma \in \Gamma$, if $\Aset \cap \gamma g^{-t} \Aset \neq \varnothing$ then $\gamma K \subseteq U_\lambda$ for some $\lambda$.
\end{theorem}

\begin{proof}
Suppose not.
Then for each $i \in \N$ there exist $\g_i \in \G$ and $t_i \to \infty$ such that $v_i \in \Aset \cap \g_i g^{-t_i} \Aset$ but $\g_i \Aset^+ \nsubseteq U_\lambda$ for all $i, \lambda$.
Since $(\g_i)$ escapes to infinity, we may assume $\g_i x \to \xi \in \bd X$ for some $\xi \in \bd X$ and $x \in X$.
This contradicts \lemref{curious lemma}.
Therefore, the theorem must hold.
\end{proof}

Putting $c = 0$ into \thmref{singular uniform expansion 2}, we obtain the following.

\begin{corollary}				\label{regular uniform expansion}
Let $X$ be a proper $\CAT(0)$ space and $\Gamma$ a group acting properly isometrically on $X$.
Let $\Aset \subset SX$ be compact, let $\Aset^- = \setp{v^-}{v \in \Aset}$ and $\Aset^+ = \setp{v^+}{v \in \Aset}$, and let $\set{U_\lambda}$ be an open cover of $\Aset^+$.
For any compact set $K \subset \bd X$ such that $\dT (\Aset^-, K) > \pi$, there is some $t_0 \ge 0$ such that for all $t \ge t_0$ and $\gamma \in \Gamma$, if $\Aset \cap \gamma g^{-t} \Aset \neq \varnothing$ then $\gamma K \subseteq U_\lambda$ for some $\lambda$.
\end{corollary}

\section{Quasi-Product Neighborhoods}

Fix a metric $\rho$ on $\bd X$ (with the cone topology).
Let $v_0 \in \Reg$, let $p = v_0(0)$, and let $\epsilon \ge 0$.
For each $\delta > 0$, let
\[\Aset(v_0, \epsilon, \delta)
= \pi_p^{-1} \Big( \cl\Bcone(v_0^-, \delta) \times \cl\Bcone(v_0^+, \delta) \times [0, \epsilon] \Big).\]
We may abbreviate $\Aset(v_0, \epsilon, \delta) = \Aed = \Aset_\delta = \Aset$.
Since $v_0 \in \Reg$, by \lemref{Ballmann's lemma} we know $\Aset_\delta$ is always compact for $\delta$ sufficiently small.
In fact, we have the following.

\begin{lemma}					\label{diam A}
Let $v_0 \in \Reg$.
For all $\epsilon \ge 0$ we have
$\lim\limits_{\delta \to 0} \diam \Aed \le 2\epsilon + \diam CS(v_0)$.
\end{lemma}

\begin{proof}
Suppose, by way of contradiction, there exist $\alpha > 0$ and sequences of $\delta_n > 0$ and $v_n, w_n \in \Aset_{\epsilon, \delta_n}$ such that $\delta_n \to 0$ but $d(v_n, w_n) \ge 2\epsilon + \diam CS(v_0) + \alpha$ for all $n$.
For each $n$ find $s_n, t_n \in [0, \epsilon]$ such that $g^{-s_n} v_n, g^{-t_n} w_n \in \Azd$.
By the triangle inequality $d(g^{-s_n} v_n, g^{-t_n} w_n) \ge \diam CS(v_0) + \alpha$ for all $n$.
We may assume $g^{-s_n} v_n \to v$ and $g^{-t_n} w_n \to w$ for some $v, w \in \bigcap_{\delta > 0} \Azd$.
Thus $v,w \in CS(v_0)$, hence $d(v, w) \le \diam CS(v_0)$, contradicting $g^{-s_n} v_n \to v$ and $g^{-t_n} w_n \to w$.
Therefore, the statement of the lemma must hold.
\end{proof}

Let $\epsilon, \delta > 0$.
For each $t \in \R$ and $\g \in \G$, let
\[\Bset^{\g}(v_0, \epsilon, \delta, t)
= \Aset(v_0, \epsilon, \delta) \cap \g g^{-t} \Aset(v_0, \epsilon, \delta).\]
We may abbreviate $B_{\g}(v_0, \epsilon, \delta, t) = B_{\g}^{\epsilon, \delta, t} = \Bset^{\g}_{\delta, t} = \Bset^{\g}$.

\begin{lemma}					\label{EBg disjoint}
Let $v_0 \in \Reg$ have zero width.
Assume $\G$ acts freely, properly discontinuously, by isometries on $X$.
There exist $\epsilon_0 > 0$ and $\delta_0 > 0$ such that for all $\epsilon \in [0,\epsilon_0]$, $\delta \in (0,\delta_0]$, and $t \in \R$,
the sets $\emap(\Bset^{\g}) = \emap(\Bset^{\g}_{\epsilon,\delta,t})$ are pairwise disjoint.
\end{lemma}

\begin{proof}
Let $p = v_0(0)$.
Because $\G$ acts freely and properly discontinuously on $X$, there is some $r_0 > 0$ such that $d(p, \g p) \ge r_0$ for all nontrivial $\g \in \G$.
Let $\epsilon_0 = \slashfrac{r_0}{12}$, and let $\delta_0 > 0$ be small enough that $\diam \Aset_{2\epsilon_0,\delta_0} < 6\epsilon_0$.
This implies $\Aset_{2\epsilon_0,\delta_0} \cap \g \Aset_{2\epsilon_0,\delta_0} = \varnothing$ for all $\g \neq \id$ by the triangle inequality.

Now let $\epsilon \in [0,\epsilon_0]$ and $\delta \in (0,\delta_0]$.
Let $\g, \g' \in \G$ be such that $\emap(\Bset^{\g}) \cap \emap(\Bset^{\g'})$ is nonempty.
By definition of $\Aset$, there is exist $t' \in \R$ and $w \in SX$ such that
$w \in g^{t} \Bset^{\g} \cap g^{t'} \Bset^{\g'}$.
Let $r = t' - t$ and $\phi = \g^{-1} \g'$.
Then
\begin{align*}
w &\in (g^{t} \Aset \cap \g \Aset) \cap (g^{t'} \Aset \cap \g' g^{r} \Aset) \\
&= (g^{t} \Aset \cap g^{t'} \Aset) \cap (\g \Aset \cap \g' g^{r} \Aset).
\end{align*}
So $w \in g^{t} \Aset \cap g^{t'} \Aset$, hence $\abs{r} \le \epsilon$ by definition of $\Aset$.
Then also
\begin{align*}
\g^{-1} w &\in \Aed \cap \phi g^{r} \Aed \\
&\subset g^{-\epsilon} \Aset_{2\epsilon_0,\delta_0} \cap \phi g^{-\epsilon} \Aset_{2\epsilon_0,\delta_0},
\end{align*}
which is empty by the previous paragraph unless $\phi = \id$.
Therefore $\g = \g'$.
\end{proof}

\begin{corollary}				%
All the $\Bset^{\g}$ are disjoint.
\end{corollary}

\begin{lemma}					\label{compact neighborhood 4}
Fix a zero-width geodesic $v_0 \in SX$.
Assume $\G$ acts freely, properly discontinuously, by isometries on $X$.
There exist $\epsilon_0 > 0$ and $\delta_0 > 0$ such that for every $\delta \in (0,\delta_0]$ and $\epsilon \in [0, \epsilon_0]$, the set $\Aset = \Aset(v_0, \epsilon, \delta)$ satisfies all the following:
\begin{enumerate}
\item
\label{neighborhood condition 2}
If $\epsilon > 0$ then $\Aset$ contains an open neighborhood of $g^{\slashfrac{\epsilon}{2}} v_0$ in $SX$.
\item
\label{compactness condition 2}
$\Aset$ is compact.
\item
\label{range of s condition 2}
For all $v \in \Aset$, $g^t v \in \Aset$ if and only if $0 \le s(g^t v) \le \epsilon$.
\item
\label{separation condition 2}
$\dT(\Aset_{\delta}^-, \Aset_{\delta}^+) > \pi$.
\item
\label{disjointness condition 2}
The sets $\emap(\Bset^{\g}) = \emap(\Bset^{\g}_{\epsilon,\delta,t})$ are pairwise disjoint for all $t \in \R$.
\end{enumerate}
\end{lemma}

\begin{proof}
Property
\itemrefstar{neighborhood condition 2}
follows from continuity of $\pi_p$,
\itemrefstar{compactness condition 2} and \itemrefstar{separation condition 2}
\lemref{Ballmann's lemma},
\itemrefstar{range of s condition 2}
the definitions, and
\itemrefstar{disjointness condition 2}
\lemref{EBg disjoint}.
\end{proof}

\begin{remark}
Only property \itemrefstar{disjointness condition 2} requires $v_0$ zero-width and $\G$ acting freely.
The others require only $v_0$ rank-one and $\G$ acting properly isometrically.
\end{remark}

\section{Mixing Calculations}

\subsection{Measures}
\label{measures}

We recall the measures constructed in \cite{ricks-mixing}.
For $\xi \in \bd X$ and $p,q \in X$, let $b_{\xi} (p, q)$ be the Busemann cocycle
\[b_{\xi} (p, q) = \lim_{t \to \infty} \left[ d([q,\xi)(t), p) - t \right].\]
These functions are $1$-Lipschitz in both variables and satisfy the cocycle property $b_{\xi} (x, y) + b_{\xi} (y, z) = b_{\xi} (x, z)$.  Furthermore, $b_{\gamma \xi} (\gamma x, \gamma y) = b_{\xi} (x, y)$ for all $\gamma \in \Isom X$.

The critical exponent
$\delta_\G = \inf \setp{s \ge 0}{\sum_{\g \in \G} e^{-s d(p, \g q)} < \infty}$
of the Poincar\'e series for $\G$ does not depend on choice of $p$ or $q$.
We shall always assume $\delta_\G < \infty$ (which holds whenever $\G$ is finitely generated, for instance).
Then Patterson's construction yields a conformal density $\family{\mu_p}_{p \in X}$ of dimension $\delta_\G$ on $\bd X$, called the \defn{Patterson-Sullivan} measure.

\begin{definition}\label{conformal density}
A \defn{conformal density of dimension $\delta$} is a family $\family{\mu_p}_{p \in X}$ of equivalent finite Borel measures on $\bd X$, supported on $\Limitset$, such that for all $p,q \in X$ and $\g \in \G$:
\begin{enumerate}
\item \label{equivariance}
the pushforward $\gamma_* \mu_p = \mu_{\gamma p}$ and
\item \label{Radon-Nikodym}
the Radon-Nikodym derivative $\frac{d\mu_q}{d\mu_p}(\xi) = e^{-\delta b_\xi (q, p)}$.
\end{enumerate}
\end{definition}

Now fix $p \in X$.
For $(v^-,v^+) \in \emap(SX)$, let $\beta_p \colon \emap(SX) \to \R$ by $\beta_p (v^-, v^+) = (b_{\xi} + b_{\eta}) (v(0), p)$; this does not depend on choice of $v \in \emap^{-1}(v^-,v^+)$.
The measure $\mu$ on $\dbX$ defined by
\[d\mu (\xi, \eta)
= e^{-\delta_\G \beta_p (\xi, \eta)} d\mu_p (\xi) d\mu_p (\eta)\]
is $\G$-invariant and does not depend on choice of $p \in X$; it is called a \emph{geodesic current}.

The \defn{Bowen-Margulis} measure $m$, a Radon measure on $SX$ that is invariant under both $g^t$ and $\G$, is constructed as follows:
The measure $\mu \times \lambda$ on $\dbX \times \R$ ($\lambda$ is Lebesgue measure) is supported on $\emap(\Zerowidth) \times \R$, where $\Zerowidth \subseteq SX$ is the set of zero-width geodesics in $X$.
Then the map $\pi_p \colon SX \to \dbX \times \R$ given by
\[\pi_p(v) = (v^-, v^+, b_{v^-} (v(0), p))\]
is seen to be a homeomorphism from $\Zerowidth$ to $\emap(\Zerowidth)$, hence $m = \mu \times \lambda$ may be viewed as a Borel measure on $SX$.
Moreover, from \cite{ricks-mixing} we have the following.

\begin{proposition}				\label{zero-width geodesics are dense}
The zero-width geodesics are dense in $SX$.
\end{proposition}

(However, the zero-width geodesics do not in general form an open set in $SX$.)

The Bowen-Margulis measure $m$ has a quotient measure $m_\G$ on $\modG{SX}$.
Since we assume $\G$ acts freely on $X$ (and therefore on $SX$), $m_\G$ can be described by saying that whenever $A \subset SX$ is a Borel set on which $\pr$ is injective, $m_\G(\pr A) = m(A)$.

One can adapt the methods of Knieper's proof \cite{knieper98} that the Bowen-Margulis measure is the unique measure of maximal entropy to the locally $\CAT(0)$ case.
One thus obtains the following theorem.

\begin{theorem}					\label{uniqueness of maximal entropy}
Let $\G$ be a group acting freely geometrically on a proper, geodesically complete $\CAT(0)$ space $X$ with rank one axis.
The Bowen-Margulis measure $m_\G$ on $\modG{SX}$ is the unique measure (up to rescaling) of maximal entropy for the geodesic flow, which has entropy $h = \delta_\G$.
\end{theorem}

To simplify notation, we write $h := \delta_\G$, even if $\G$ does not act cocompactly.

The $\G$-action on $X$ is said to have \defn{arithmetic length spectrum} if the translation lengths of axes are all contained in some discrete subgroup $c\Z$ of $\R$.
In \cite{ricks-mixing}, we showed that when $\Limitset = \bd X$, $X$ is geodesically complete, and $m_\G$ is finite, the only examples of arithmetic length spectrum are when $X$ is a tree with integer edge lengths, up to homothety.
Moreover, when the $\G$-action on $X$ does not have arithmetic length spectrum, the measure $m_\G$ is mixing under the geodesic flow $g_\G^t$.

STANDING HYPOTHESIS:
We assume throughout that $m_\G$ is finite, and thus we may normalize the measure by assuming $m_\G (\modG{SX}) = 1$.
We also assume non-arithmetic length spectrum, so $m_\G$ is mixing.

\subsection{Averaging}
\label{averaging}

Fix a zero-width geodesic $v_0 \in SX$, and let $p = v_0 (0)$.
Let $\epsilon \in (0, \epsilon_0]$ and $\delta \in (0, \delta_0]$.  

Our goal in this section is to prove \corref{mu(EB)}, which describes the total measure of intersections $\Aset \cap \G g^t(\Aset)$ for large $t$.
It is easy to see by mixing that $\lim_{t \to \infty} m(\Bset) = m(\Aset)^2$.
Less obvious, however, is that
$\lim_{t \to \infty} \mu(\emap(\Bset)) = \frac{2}{\epsilon} m(\Aset)^2$.

\begin{definition}
Define $s \colon SX \to \R$ by $s(v) = b_{v^-} (v(0), p)$,
and $\tau_{\g} \colon SX \to \R$ by $\tau_{\g} (v) = b_{v^-} (\g p, p) - t$.
\end{definition}

\begin{lemma}					\label{tau}
$\tau_{\g} (v) = s(v) - s (\g^{-1} g^t v)$.
\end{lemma}

\begin{proof}
We compute
\begin{align*}
s(v) - s(\g^{-1} g^t v)
&= b_{v^-}(v(0), p) - b_{\g^{-1} v^-}(\g^{-1} v(t), p) \\
&= b_{v^-}(v(0), p) - \left[ b_{v^-}(v(0), \g p) + t \right] \\
&= b_{v^-} (\g p, p) - t \\
&= \tau_{\g} (v).
\qedhere
\end{align*}
\end{proof}

Let $\Bset = \bigcup_{\g \in \G} \Bset^{\g}$.

Define $\phi \colon \Bset \to \G$ by putting $\phi(v)$ equal to the unique $\g \in \G$ such that $v \in \Bset^{\g}$.

Define $\tau \colon \Bset \to \R$ by $\tau(v) = \tau_{\phi(v)} (v)$ and
$\ell \colon \Bset \to \R$ by $\ell(v) = \epsilon - \abs{\tau (v)}$.

\begin{lemma}					\label{length}
$\ell(v)$ is the length of the geodesic segment $g^{\R}(v) \cap \Bset$.
\end{lemma}

\begin{proof}
This follows from \lemref{tau}, by \itemrefstar{range of s condition 2} and \itemrefstar{disjointness condition 2} of \lemref{compact neighborhood 4}.
\end{proof}

\begin{corollary}				\label{mu and m}
For all $f \in L^1(\mu)$,
\[\int_{\emap(\Bset)} f \; d\mu
= \int_{\Bset} \frac{f \circ \emap}{\ell} \; dm.\]
\end{corollary}

By \lemref{compact neighborhood 4} \itemrefstar{disjointness condition 2}, the map $\phi \colon \Bset \to \G$ factors as $\phi = \hat \phi \circ \emap$ for some $\hat \phi \colon \emap(\Bset) \to \G$.
Similarly, $\tau \colon \Bset \to \R$ factors as $\tau = \hat \tau \circ \emap$ for some $\hat \tau \colon \emap(\Bset) \to \R$,
and $\ell = \hat \ell \circ \emap$.

\begin{corollary}				\label{mu and m with tau}
For all $f \in L^1(\R)$,
\[\int_{\emap(\Bset)} f \circ \hat \tau \; d\mu
= \int_{\Bset} \frac{f \circ \tau}{\ell} \; dm.\]
\end{corollary}

Define $\sigma \colon \Bset \to \R$ by $\sigma(v) = s (\phi(v)^{-1} g^t v)$.

\begin{lemma}					%
$\sigma$ is continuous.
\end{lemma}

\begin{proof}
The restriction of $\sigma$ to each $\Bset^{\g}$ is $s \circ \g^{-1} \circ g^t$, and $\Bset$ is the disjoint union of finitely many (closed) $\Bset^{\g}$.
\end{proof}

\begin{fact}					%
$\tau = s - \sigma$.
\end{fact}

\begin{fact}					%
Both $s(v), \sigma(v) \in [0, \epsilon]$ for all $v \in \Bset$.
\end{fact}

\begin{lemma}					\label{mixing pushforwards}
Let $\fn \colon \modG{SX} \to \R$ be measurable, and let
$\fn_t = \fn \circ g^t_\G$.
Then
\[\lim_{t \to \infty} ( \fn \times \fn_t )_* m_\G (C \times D)
=
( \fn_* m_\G \times \fn_* m_\G )(C \times D)\]
for every measurable $C \times D \subseteq \R^2$.
\end{lemma}

\begin{proof}
By mixing,
\(\lim\limits_{t \to \infty} m_\G ( \psi^{-1}(C) \cap \psi_t^{-1}(D) )
=
m_\G(\psi^{-1} (C)) \cdot m_\G(\psi^{-1} (D))\).
\end{proof}

\begin{lemma}					\label{mixing for B}
If $f \colon [0, \epsilon] \times [0, \epsilon] \to \R$ is Riemann integrable, then
\[\lim_{t \to \infty} \int_{\Bset} f( s(v), \sigma(v) ) \; dm(v)
=
\frac{m(\Aset)^2}{\epsilon^2} \int_{0}^{\epsilon} \int_{0}^{\epsilon} f(x,y) \; dx \, dy.\]
\end{lemma}

\noindent
Thus $(s \times \sigma)_* m$ converges weakly to $%
\frac{m(\Aset)^2}{\epsilon^2}$ times Lebesgue measure on $[0,\epsilon]^2$.

\begin{proof}
Since $s_* m$ is $\frac{m(\Aset)}{\epsilon}$ times Lebesgue measure on $[0,\epsilon]$,
by \lemref{mixing pushforwards}
the conclusion of the theorem holds whenever $f$ is the characteristic function of a measurable product set $C \times D \subseteq [0, \epsilon]^2$.
This easily extends to all finite linear combinations of characteristic functions of measurable product sets.

Now if $f$ is Riemann integrable, there exist step functions $\varphi_n \le f \le \psi_n$ satisfying $\lim_n \int_{0}^{\epsilon} \! \int_{0}^{\epsilon} \varphi_n = \lim_n \int_{0}^{\epsilon} \! \int_{0}^{\epsilon} \psi_n = \int_{0}^{\epsilon} \! \int_{0}^{\epsilon} f$.
Then
\[\int \varphi_n \; d(s \times \sigma)_* m
\, \le \int f \; d(s \times \sigma)_* m
\, \le \int \psi_n \; d(s \times \sigma)_* m\]
and so letting $t \to \infty$ we obtain
\begin{gather*}
\frac{m(\Aset)^2}{\epsilon^2} \int_{0}^{\epsilon} \! \int_{0}^{\epsilon} \varphi_n
\le \liminf_{t \to \infty} \int f \; d(s \times \sigma)_* m \\
\le \limsup_{t \to \infty} \int f \; d(s \times \sigma)_* m
\le \frac{m(\Aset)^2}{\epsilon^2} \int_{0}^{\epsilon} \! \int_{0}^{\epsilon} \psi_n.
\end{gather*}
Letting $n \to \infty$ we find
\begin{equation*}
\lim_{t \to \infty} \int f \; d(s \times \sigma)_* m
= \frac{m(\Aset)^2}{\epsilon^2} \int_{0}^{\epsilon} \! \int_{0}^{\epsilon} f.
\qedhere
\end{equation*}
\end{proof}

\begin{lemma}					\label{change of variables}
If $f \colon (-\epsilon, \epsilon) \to \R$ is Riemann integrable, then the function
\(\tilde{f} \colon (0, \epsilon) \times (0, \epsilon) \to \R\)
given by
\(\tilde{f} (x,y) = \frac{1}{\epsilon - \abs{x - y}} f(x - y)\)
is Riemann integrable, and
\[\int_{0}^{\epsilon} \! \int_{0}^{\epsilon} \tilde{f}(x,y) \; dx \, dy
= \int_{-\epsilon}^{\epsilon} f(z) \; dz.\]
\end{lemma}

\begin{proof}
By change of variables (putting $z = x - y$ and $w = x + y$),
\[\int_{0}^{\epsilon} \! \int_{0}^{\epsilon} \tilde{f}(x,y) \; dx \, dy
= \int_{-\epsilon}^{\epsilon} \frac{1}{2} \int_{-\epsilon + \abs{z}}^{\epsilon - \abs{z}} \frac{f(z)}{\epsilon - \abs{z}} \; dw \, dz
= \int_{-\epsilon}^{\epsilon} f(z) \; dz.
\qedhere\]
\end{proof}

\begin{remark}
In the notation of \lemref{change of variables},
$\frac{f \circ \tau}{\ell} = \tilde{f} \circ (s \times \sigma)$.
\end{remark}

\begin{proposition}				\label{integration}
Let $X$ be a proper $\CAT(0)$ space.
Assume $\G$ acts freely, properly discontinuously, and by isometries on $X$, and that $m_\G$ is finite and mixing.
If $f \colon (-\epsilon, \epsilon) \to \R$ is Riemann integrable and $\tilde{f}$ is as in \lemref{change of variables}, then
\[\lim_{t \to \infty} \int_{\emap(\Bset)} f \circ \hat \tau \; d\mu
= \lim_{t \to \infty} \int_{\Bset} \frac{f \circ \tau}{\ell} \; dm
=
\frac{m(\Aset)^2}{\epsilon^2} \int_{0}^{\epsilon} \! \int_{0}^{\epsilon} \tilde{f}
= \frac{m(\Aset)^2}{\epsilon^2} \int_{-\epsilon}^{\epsilon} f.\]
\end{proposition}

\begin{proof}
The first equality follows from \corref{mu and m with tau},
the last equality from \lemref{change of variables},
and the middle asymptotic from \lemref{mixing for B} because
\(\frac{f \circ \tau}{\ell} = \tilde{f} \circ (s \times \sigma)\).
\end{proof}

\begin{corollary}				\label{mu(EB)}
$\lim_{t \to \infty} \mu(\emap(\Bset)) = \frac{2}{\epsilon} m(\Aset)^2= \lim_{t \to \infty} \frac{2}{\epsilon} m(\Bset)$.
\end{corollary}

\begin{proof}
Putting $f = 1$ in \propref{integration}, we obtain
$\lim_{t \to \infty} \mu(\emap(\Bset))
= \frac{m(\Aset)^2}{\epsilon^2} \int_{-\epsilon}^{\epsilon} 1
= \frac{2}{\epsilon} m(\Aset)^2$.
Putting $\tilde f = 1$, we find
$\frac{2}{\epsilon} m(\Aset)^2
= \lim_{t \to \infty} \frac{2}{\epsilon} m(\Bset)$ because
\(\frac{f \circ \tau}{\ell} = \tilde{f} \circ (s \times \sigma)\).
\end{proof}

\begin{remark}
In terms of averages, we find
\(\lim_{t \to \infty} \frac{1}{\mu(\emap(\Bset))} \int_{\emap(\Bset)} f \circ \hat \tau \; d\mu
= \frac{1}{2 \epsilon} \int_{-\epsilon}^{\epsilon} f\).
In particular,
\[\lim_{t \to \infty} \frac{1}{\mu(\emap(\Bset))} \int_{\emap(\Bset)} \hat \ell \; d\mu
= \frac{\epsilon}{2}
\quad \text{and} \quad
\lim_{t \to \infty} \frac{1}{m(\Bset)} \int_{\Bset} \ell \; dm
= \frac{2\epsilon}{3}.\]
\end{remark}

\section{Product Estimates}

For this section, fix $v_0 \in SX$ and $\epsilon, \delta > 0$, and let $t \in \R$.

\begin{lemma}					\label{conformal density calculation 2}
Let $U,V \subseteq \bd X$ be Borel sets and let $\g \in \G$.
Assume $\g V \subseteq V$ and $\abs{\beta_p} \le C$ on $U \times V$.
Then
\begin{equation*}
e^{-2hC}
\le \frac%
{\int_{U \times \g V} f(\xi,\eta) \; d\mu(\xi,\eta)}%
{\int_{U \times V} f(\xi,\g \eta') \, e^{-h b_{\g \eta'} (p, \g p)} \; d\mu(\xi,\eta')}
\le e^{2hC}.
\end{equation*}
\end{lemma}

\begin{proof}
By the properties of conformal densities and the definition of $\mu$,
\begin{align*}
\int_{U \times \g V} f(\xi,\eta) \, d\mu(\xi,\eta)
&= \int_{U \times \g V} f(\xi,\eta) \, e^{-h \beta_p(\xi,\eta)} \, d\mu_p(\xi) \, d\mu_p(\eta) \\
&= \int_{U \times V} f(\xi,\g \eta') \, e^{-h \beta_p(\xi,\g \eta')} \, d\mu_p(\xi) \, d\mu_{\g^{-1}p}(\eta') \\
&= \int_{U \times V} f(\xi,\g \eta') \, e^{-h \beta_p(\xi,\g \eta')} \, d\mu_p(\xi) \, e^{-h b_{\eta'} (\g^{-1} p, p)} \, d\mu_{p}(\eta') \\
&= \int_{U \times V} f(\xi,\g \eta') \, e^{-h b_{\g \eta'} (p, \g p)} e^{-h \beta_p(\xi,\g \eta')} \, d\mu_p(\xi) \, d\mu_{p}(\eta') \\
&= \int_{U \times V} f(\xi,\g \eta') \, e^{-h b_{\g \eta'} (p, \g p)} e^{-h [\beta_p(\xi,\g \eta') - \beta_p(\xi,\eta')]} \, d\mu(\xi,\eta').
\end{align*}
The conclusion of the lemma follows immediately.
\end{proof}

We will use \lemref{conformal density calculation 2} with $U \times V = \Aset_\delta^- \times \Aset_\delta^+$.
By Lemma 5.3 of \cite{ricks-mixing}, $\beta_p$ is continuous on $\RE$.
Thus $\lim\limits_{\delta \to 0} \max\limits_{v \in \Aset_{\delta}} \abs{\beta_p(v)} = 0$.
However, for simplicity we will just use the bound
$\max\limits_{v \in \Aset_{\delta}} \abs{\beta_p(v)} \le 2 \diam \pi (\Azd) \le 2 \diam (\Aed)$.

\begin{definition}
Let $I = I(v_0, \epsilon, \delta, t)$ be the set of nontrivial $\g \in \G$ such that $\Bset^\g = \Aset \cap g^{-t} \g \Aset$ is not empty.
Call $\g \in I$ \defn{unclipped} if $\g \Aset^+ \subseteq \Aset^+$ and $\Aset^- \subseteq \g \Aset^-$.
Equivalently, $\emap(\Bset^\g) = \Aset^- \times \g \Aset^+$.
Let $\Iu$ be the set of unclipped $\g \in I$.
\end{definition}

\begin{lemma}					\label{unclipped measure 2}
Let $\g \in \G$ be unclipped.
Then
\begin{equation*}
e^{-6h \diam(\Aset)}
\le \frac%
{e^{ht} \int_{\emap(\Bset^\g)} f(\xi,\eta) \; d\mu(\xi,\eta)}%
{\int_{\emap(\Aset)} f(\xi,\g \eta') \, d\mu(\xi,\eta')}
\le e^{6h \diam(\Aset)}.
\end{equation*}
\end{lemma}

\begin{proof}
By \lemref{conformal density calculation 2},
\begin{equation*}
e^{-4h \diam(\Aset)}
\le \frac%
{e^{ht} \int_{\Aset^- \times \g \Aset^+} f(\xi,\eta) \; d\mu(\xi,\eta)}%
{\int_{\Aset^- \times \Aset^+} f(\xi,\g \eta') \, e^{h[t - b_{\g \eta'} (p, \g p)]} \, d\mu(\xi,\eta')}
\le e^{4h \diam(\Aset)}.
\end{equation*}
If $\eta' \in \Aset^+$ then $\g \eta' = w^+$ for some $w \in \Bset^\g$ because $\g$ is unclipped.
So both $w$ and $\g^{-1} g^t w$ are in $\Aset$.
Hence
\begin{align*}
\abs{b_{\g \eta'} (p, \g p) - t}
&= \abs{b_{w^+} (p, \g p) - t} \\
&\le \abs{b_{w^+} (w(0), \g \cdot \g^{-1} g^t w(0)) - t} + 2 \diam(\pi(\Aset)) \\
&= 2 \diam \pi(\Aset)
\le 2 \diam(\Aset).
\end{align*}
Therefore,
\begin{equation*}
e^{-6h \diam(\Aset)}
\le \frac%
{e^{ht} \int_{\Aset^- \times \g \Aset^+} f(\xi,\eta) \; d\mu(\xi,\eta)}%
{\int_{\Aset^- \times \Aset^+} f(\xi,\g \eta') \, d\mu(\xi,\eta')}
\le e^{6h \diam(\Aset)}.
\qedhere
\end{equation*}
\end{proof}

\begin{definition}
To simplify future statements, we write $\Ced = e^{6h \diam (\Aed)}$.
\end{definition}

Notice that for $\epsilon > 0$ fixed, $\Ced$ is an upper semicontinuous increasing function of $\delta$.
And for $\delta > 0$ fixed, $\Ced$ is a continuous increasing function of $\epsilon$.

\begin{corollary}				\label{unclipped measure 3}
Let $\g \in \Iuedt$.
Then
\begin{equation*}
\frac{1}{\Ced}
\le \frac{e^{ht}}{\mu(\emap(\Aset))} \int_{\emap(\Bset^\g)} d\mu
\le \Ced.
\end{equation*}
\end{corollary}

\section{Counting Unclipped Intersections}

Fix a zero-width geodesic $v_0 \in SX$.
Let $N = N(v_0, \epsilon, \delta, t) = \card{I(v_0, \epsilon, \delta, t)}$ and $\Nu = \Nu(v_0, \epsilon, \delta, t) = \card \Iu(v_0, \epsilon, \delta, t)$.

\begin{lemma}					\label{unclipped counting}
Assume $\epsilon \in (0, \epsilon_0]$ and $\delta \in (0, \delta_0]$.
Then
\begin{equation*}
\frac{1}{\Ced}
\le e^{-ht} \frac{\mu(\emap(\Aset))}{\mu(\emap(\Bu))} \Nu
\le \Ced.
\end{equation*}
\end{lemma}

\begin{proof}
Start with the identity
\begin{equation*}
\Nu
= \sum_{\g \in \Iu} 1
= \sum_{\g \in \Iu}
\frac{1}{\mu(\emap(\Bset^\g))} \int_{\emap(\Bset^\g)} d\mu.
\end{equation*}
By \corref{unclipped measure 3},
\begin{equation*}
\frac{1}{\Ced}
\le e^{-ht} \frac{\mu(\emap(\Aset))}{\mu(\emap(\Bset^\g))}
\le \Ced
\end{equation*}
for $\g$ unclipped,
so
\begin{align*}
\frac{1}{\Ced} \mu(\emap(\Bu))
&= \! \sum_{\g \in \Iu} \frac{1}{\Ced} \int_{\emap(\Bset^\g)} d\mu \\
&\le e^{-ht} \mu(\emap(\Aset)) \Nu \\
&\le \! \sum_{\g \in \Iu} \Ced \int_{\emap(\Bset^\g)} d\mu \\
&= \Ced \mu(\emap(\Bu)).
\qedhere
\end{align*}
\end{proof}

\section{Jiggling Near Rank One Geodesics}

Clearly the inclusions
\(\Iu_{\delta, t} \subseteq I_{\delta, t}\)
and
\(\Bu_{\delta, t} \subseteq \Bset_{\delta, t}\)
always hold.
We now prove inclusions when we allow $\delta > 0$ to vary.

\begin{lemma}					\label{jiggling with I}
Let $v_0 \in \Reg$ and $0 < r < \delta \le \delta_0$.
There exists $t_0 \ge 0$ such that
\[I_{r, t} \subseteq \Iu_{\delta, t} \subseteq I_{\delta, t}\]
for all $t \ge t_0$ and $\epsilon \in (0, \epsilon_0]$.
\end{lemma}

\begin{proof}
Let $\alpha = \delta - r > 0$.
By \corref{regular uniform expansion},
there exists $t_1 \ge 0$ such that for all $t \ge t_1$ and $\g \in I_{r,t}$ (i.e. $\Aset_r \cap \gamma g^{-t} \Aset_r \neq \varnothing$),
$\g \Aset_\delta^+
\subseteq \Bcone (\Aset_r^+, \alpha)
= \Aset_\delta^+$.
Similarly,
there exists $t_2 \ge 0$ such that for all $t \ge t_2$ and $\g \in I_{r,t}$ (i.e. $\Aset_r \cap \gamma^{-1} g^t \Aset_r \neq \varnothing$),
$\g^{-1} \Aset_\delta^-
\subseteq \Bcone (\Aset_r^-, \alpha)
= \Aset_\delta^-$.
So for all $t \ge t_0 = \max \set{t_1, t_2}$, if $\g \in I_{r,t}$ then $\g \in I^{unclipped}_{\delta,t}$.
\end{proof}

\begin{corollary}				\label{jiggling 2}
Let $v_0 \in \Reg$ and $0 < r < \delta \le \delta_0$.
There exists $t_0 \ge 0$ such that
\[\Bset_{r, t} \subseteq \Bu_{\delta, t} \subseteq \Bset_{\delta, t}\]
for all $t \ge t_0$ and $\epsilon \in (0, \epsilon_0]$.
\end{corollary}

In the sequel, we shall often want to state things for both $\liminf$ and $\limsup$.
The following definition makes this more convenient:
Write $a \le \fakelim_{t \to \infty} f(t) \le b$ if
for every $\epsilon > 0$ there exists $t_0 \in \R$ such that $a - \epsilon \le f(t) \le b + \epsilon$ for all $t \ge t_0$.
In other words, $\liminf_{t \to \infty} f(t) \ge a$ and $\limsup_{t \to \infty} f(t) \le b$.

\begin{lemma}					\label{counting I 3}
Let $v_0 \in SX$ be zero-width and $\epsilon \in (0, \epsilon_0]$.
Let $\delta \in (0, \delta_0]$ be a point of continuity of the nondecreasing function $r \mapsto m(\Aset_r)$.
Then
\begin{equation*}
\frac{1}{\Ced}
\le \fakelim_{t \to \infty} \frac{\Nu_{\delta,t}}{2 e^{ht} m(\Aset_{\delta})}
\le \Ced.
\end{equation*}
\end{lemma}

\begin{proof}
By \corref{mu(EB)}, 
\(\lim_{t \to \infty} \mu(\emap(\Bset_{r,t}))
= \frac{2}{\epsilon} m(\Aset_r)^2\)
for all $r \in (0, \delta_0]$.
Hence $\delta$ is a point of continuity of the function
$f(r) = \lim_{t \to \infty} \mu(\emap(\Bset_{r,t}))$.
So by \corref{jiggling 2},
\[\lim_{t \to \infty} \mu(\emap(\Bu_{\delta,t}))
= \lim_{t \to \infty} \mu(\emap(\Bset_{\delta,t}))
= \frac{2}{\epsilon} m(\Aset_{\delta})^2.\]
But now
\begin{equation*}
\frac{1}{\Ced}
\le \fakelim_{t \to \infty} \frac{\Nu_{\delta,t}}{2 e^{ht} m(\Aset_{\delta})}
\le \Ced
\end{equation*}
by \lemref{unclipped counting}
because $\mu(\emap(\Aset)) = \frac{m(\Aset)}{\epsilon}$.
\end{proof}

\begin{remark}
The points of continuity of $r \mapsto m(\Aset_r) = \epsilon \cdot \mu(\Aset_r^- \times \Aset_r^+)$ do not depend on $\epsilon$.
Also, for such $r$ we find that $\Aset_r$ is a continuity set for $m$ (that is, the topological frontier $\bd \Aset_r$ of $\Aset_r$ has $m(\bd \Aset_r) = 0$); this is easy to see because the projection $SX \to \dbX \times \R$ is continuous, and therefore $\bd \Aset_r \subseteq \bd \emap(\Aset_r) \times \set{0,\epsilon}$.
\end{remark}

\begin{lemma}					\label{counting I 6}
Let $v_0 \in SX$ be zero-width and $\epsilon \in (0, \epsilon_0]$.
Let $\delta \in (0, \delta_0)$ be a point of continuity of the nondecreasing function $r \mapsto m(\Aset_r)$.
Then
\begin{equation*}
\frac{1}{\Ced}
\le \fakelim_{t \to \infty} \frac{N_{\delta,t}}{2 e^{ht} m(\Aset_{\delta})}
\le \Ced.
\end{equation*}
\end{lemma}

\begin{proof}
Whenever $\delta' \in (\delta, \delta_0]$, we find
\(\Nu_{\delta,t} \le N_{\delta,t} \le \Nu_{\delta'\!,t} \le N_{\delta'\!,t}\)
for all $t$ sufficiently large by \lemref{jiggling with I}, hence
\[\phi(r) = \limsup_{t \to \infty} e^{-ht} \Nu_{r,t}
\quad \text{and} \quad
\psi(r) = \limsup_{t \to \infty} e^{-ht} N_{r,t}\]
satisfy $\phi(\delta) \le \psi(\delta) \le \phi(\delta') \le \psi(\delta')$.
Taking a decreasing sequence $\delta'_n \to \delta$ such that each $\delta'_n > \delta $ is a point of continuity of $r \mapsto m(\Aset_r)$, we find 
by \lemref{counting I 3} that
\[\frac{2 m(\Aset_{\delta})}{\Ced}
\le \liminf_{t \to \infty} e^{-ht} \Nu_{\delta,t}
\le \liminf_{t \to \infty} e^{-ht} N_{\delta,t}\]
and
\[\psi(\delta)
\le \liminf_{n \to \infty} \phi(\delta'_n)
\le \liminf_{n \to \infty} 2 m(\Aset_{\delta}) \Cedpn
= 2 m(\Aset_{\delta}) \Ced.\qedhere\]
\end{proof}

\section{Counting Periodic Intersections}

\begin{definition}
Let $v_0 \in \Reg$ and $\epsilon, \delta > 0$.
Define
\begin{align*}
\Iperedt &= \setp{\g \in \Iedt}{\g \text{ has an axis in } \Aed} \\
&= \setp{\g \in \G}{\g \text{ has an axis in } \Aed \text{ and } \length{\g} \in [\tpm]}
\end{align*}
and $\Gedt = \card{\Iperedt}$.
\end{definition}

Clearly the inclusion $\Iperedt \subseteq \Iedt$ always holds.

\begin{lemma}					\label{unclipped is periodic}
Let $v_0 \in SX$ be zero-width, and let $\epsilon \in (0, \epsilon_0]$ and $\delta \in (0, \delta_0]$.
Then
$\Iuedt \subseteq \Iperedt$ for all $t > 0$.
\end{lemma}

\begin{proof}
Let $\g \in \Iuedt$.
Since $\g \Aset^+ \subseteq \Aset^+$, the nested intersection
$\bigcap_{n \in \N} \g^n \Aset^+$
of compact sets must contain a point $\xi \in \bd X$.
Similarly the nested intersection
$\bigcap_{n \in \N} \g^{-n} \Aset^-$
must contain a point $\eta \in \bd X$.
Then $\xi \in \Aset^+$ and $\eta \in \Aset^-$ must be the endpoints of an axis for $\g$.
Because $\emap(\Aset) = \Aset^- \times \Aset^+$, $\Aset$ contains an axis for $\g$.
\end{proof}

\begin{proposition}				\label{good bounds for Gedt}
Let $X$ be a proper $\CAT(0)$ space.
Assume $\G$ acts freely, properly discontinuously, and by isometries on $X$, and that $m_\G$ is finite and mixing.
Let $v_0 \in SX$ be zero-width, and let $\epsilon \in (0, \epsilon_0]$.  
Let $\delta \in (0, \delta_0)$ be a point of continuity of the nondecreasing function $r \mapsto m(\Aset_r)$.
Then
\begin{equation*}
\frac{1}{\Ced}
\le \fakelim_{t \to \infty} \frac{\Gedt}{2e^{ht} m(\Aed)}
\le \Ced.
\end{equation*}
\end{proposition}

\begin{proof}
By \lemref{unclipped is periodic},
$\Nuedt \le \Gedt \le \Nedt$
for all sufficiently large $t$,
hence
\[\liminf_{t \to \infty} \frac{\Nuedt}{2e^{ht} m(\Aed)} \le \fakelim_{t \to \infty} \frac{\Gedt}{2e^{ht} m(\Aed)} \le \limsup_{t \to \infty} \frac{\Nedt}{2e^{ht} m(\Aed)}.\]
Now apply the bounds from \lemref{counting I 3} and \lemref{counting I 6}.
\end{proof}

\section{Conjugacy Classes and Intersection Segments}

For this section, we assume $\G$ acts freely, properly discontinuously, by isometries on $X$.
A non-identity element $\g \in \G$ is called \defn{axial} if there exist $v \in SX$ and $t > 0$ such that $\g v = g^t v$.

\subsection{Conjugacy Classes}

Let $\Conj(\G)$ be the set of axial conjugacy classes $[\g]$ of $\G$.
Call a function $\axis \colon \Conj(\G) \to SX$ a \defn{choice of axis} if every $\axis[\g]$ is an axis for some $\g' \in [\g]$.
In other words, for every axial $\g \in \G$ there exists $\phi \in \G$ such that $\phi \axis[\g]$ is an axis for $\g$.

Call a conjugacy class $[\g] \in \Conj(\G)$ \defn{primitive} if $\g = \phi^n$ for some $\phi \in \G$ and $n > 1$; note this does not depend on choice of representative $\g$ for $[\g]$.
Let $\Conjp(\G) \subset \Conj(\G)$ be the set of conjugacy classes which are not primitive.

For any subset $U \subseteq SX$, write $\Conjinset{U}(\G) \subseteq \Conj(\G)$ for the set of conjugacy classes $[\g]$ such that $\g$ has an axis in $\G U$; this also does not depend on choice of representative $\g$ for $[\g]$.
Also define $\Conjpinset{U}(\G) = \Conjp(\G) \cap \Conjinset{U}(\G)$.

For $v \in SX$, let $\length{v}$ be the length of the smallest period under $g^t_\G$ of the projection $\pr(v) \in \modG{SX}$, with $\length{v} = \infty$ if $\pr(v)$ is not periodic.

For $\g \in \G$, let $\length{\g}$ be the translation length of $\g$.

For $t \ge t' \ge 0$, let
$\ConjG (t',t) = \setp{[\g] \in \Conj(\G)}{t' \le \length{\g} \le t}$.
Similarly define $\ConjGp (t',t)$, $\ConjGinset{U} (t',t)$, and $\ConjGpinset{U} (t',t)$ for $U \subseteq SX$.
Let $\CG (t',t) = \card{\ConjG (t',t)}$, and similarly define $\CGp (t',t)$, $\CGinset{U} (t',t)$, and $\CGpinset{U} (t',t)$.

\subsection{Intersection Segments}

Let $v_0 \in SX$, $\epsilon \in (0,\epsilon_0]$, and $\delta \in (0,\delta_0]$.

For every $v \in SX$, the intersection of $\G \Aed$ with $g^{\R} v$ is the disjoint union of orbit segments of length $\epsilon$.
Call these segments \defn{intersection segments for $v$ with $\Aed$}; call two segments \defn{equivalent} if there is an isometry $\g \in \G$ carrying one to the other.

Let $\segclassesA(v)$ be the collection of equivalence classes of intersection segments for $v$ with $\Aed$, and let $\segcountA(v) = \card{\segclassesA(v)}$.
Notice that $\segclassesA(v)$ is in natural bijection with the collection of disjoint orbit segments (length $\epsilon$) arising as intersections of $\Aed$ with $\G g^{\R} v$.
Immediately we deduce the following.

\begin{lemma}					\label{Gedt via segcountA}
For all $U$ satisfying $\Aed \subseteq U \subseteq SX$, we have
\begin{equation*}
\Gedt = \sum_{[\g] \in \ConjGinset{U} (\tpm)} \segcountA(\axis[\g]).
\end{equation*}
\end{lemma}

\begin{proof}
By construction of $\Aed$, if $v \in \Aed$ and $w \parallel v$ then $g^t w \in \Aed$ for some $t \in \R$.
So 
$\Gedt$ is the number of $\g \in \G$ with an axis in $\Aed$ such that $\length{\g} \in [\tpm]$,
while on the other hand
$\ConjGinset{U} (\tpm)$ is the set of $[\g] \in \Conj(\G)$ such that $\g$ has a conjugate with an axis in $U$ and $\length{\g} \in [\tpm]$,
and
$\segcountA(\axis[\g])$ is the number of conjugates of $\g$ with an axis in $\Aed$.
\end{proof}

\begin{lemma}					\label{early upper bound for ConjGA}
Let $v_0 \in SX$ be zero-width, and let $\epsilon \in (0, \epsilon_0]$.  
Let $\delta \in (0, \delta_0)$ be a point of continuity of the nondecreasing function $r \mapsto m(\Aset_r)$.
Then
\begin{equation*}
\limsup\limits_{t \to \infty} \frac{\CGA (\tpm)}{2 e^{ht} m(\Aed)} \le \Ced.
\end{equation*}
\end{lemma}

\begin{proof}
Since $\segcountA(\axis[\g]) \ge 1$ for all $[\g] \in \ConjGA$,
we have $\CGA(\tpm) \le \Gedt$ by \lemref{Gedt via segcountA}.
Apply the bounds from \propref{good bounds for Gedt}.
\end{proof}

\section{Measuring along Periodic Orbits}

For each $v \in SX$, let $\lv$ be Lebesgue measure on $g^{\R}v$.
Notice the quotient measure $\lvG$ on $\modG{SX}$ has $\norm{\lvG} = \length{v}$.

\begin{lemma}					\label{segcountA(v) via lvG}
Let $v_0 \in SX$, $\epsilon \in (0,\epsilon_0]$, and $\delta \in (0,\delta_0]$.
For all $v \in SX$, there are $\frac{1}{\epsilon} \lvG (\pr \Aed)$ equivalence classes of intersection segments for $v$ with $\Aed$; that is,
\[\segcountA(v) = \frac{1}{\epsilon} \lvG (\pr \Aed).\]
\end{lemma}

\begin{proof}
The intersection segments for $v$ with $\Aed$ are each of length $\epsilon$, and they are pairwise disjoint.
Hence $\lvG (\pr \Aed) = \epsilon \cdot \segcountA(v)$.
\end{proof}

For any $U \subseteq SX$ and $t, \alpha > 0$, define
\begin{align*}
\lmatainset{U} &= \frac{1}{\CGinset{U} (\toma)} \sum_{[\g] \in \ConjGinset{U} (\toma)} \frac{1}{\smallnorm{\laG}} \laG \\
\lpatainset{U} &= \frac{1}{\CGpinset{U} (\toma)} \sum_{[\g] \in \ConjGpinset{U} (\toma)} \frac{1}{\smallnorm{\laG}} \laG \\
\aaltlmatainset{U} &= \frac{1}{t \cdot \CGinset{U} (\tpma)} \sum_{[\g] \in \ConjGinset{U} (\tpma)} \laG.
\end{align*}

\begin{lemma}					\label{Gedt via aaltlmateU}
For all $U$ satisfying $\Aed \subseteq U \subseteq SX$, we have
\[\Gedt = \frac{t}{\epsilon} \CGinset{U}(\tpm) \aaltlmateinset{U} (\Aed).\]
\end{lemma}

\begin{proof}
Combining \lemref{Gedt via segcountA} and \lemref{segcountA(v) via lvG}, we obtain
\begin{equation*}
\Gedt = \sum_{[\g] \in \ConjGinset{U} (\tpm)} \segcountA(\axis[\g])
= \frac{t}{\epsilon} \CGinset{U} (\tpm) \aaltlmate (\Aed).
\qedhere
\end{equation*}
\end{proof}

From \propref{good bounds for Gedt} and \lemref{Gedt via aaltlmateU} we obtain the following.

\begin{corollary}				\label{ConjGU via aaltlmateU}
Let $v_0 \in SX$ be zero-width, and let $\epsilon \in (0, \epsilon_0]$.  
Let $\delta \in (0, \delta_0)$ be a point of continuity of the nondecreasing function $r \mapsto m(\Aset_r)$.
Then
\begin{equation*}
\frac{1}{\Ced} \cdot \frac{m(\Aed)}{\aaltlmateinset{U}(\Aed)} \le \fakelim_{t \to \infty} \frac{\CGinset{U}(\tpm)}{2\epsilon e^{ht}/t} \le \Ced \cdot \frac{m(\Aed)}{\aaltlmateinset{U}(\Aed)}
\end{equation*}
whenever $\Aed \subseteq U \subseteq SX$.
\end{corollary}

The measures $\aaltlmatainset{U}$ and $\lmattainset{U}$ have the same weak limits.
In fact, one easily checks the following, directly from the definitions.

\begin{lemma}					\label{lataU and aaltlataU are close in norm}
Let $U \subseteq SX$ be such that $\Aed \subseteq U$.
For any fixed $\alpha > 0$ and choice of axis $\axis$,
\(\lim\limits_{t \to \infty} \norm{\aaltlmatainset{U} - \lmattainset{U}} = 0\).
\end{lemma}

\begin{corollary}				\label{early lower bound for ConjGA}
Let $v_0 \in SX$ be zero-width, and let $\epsilon \in (0, \epsilon_0]$.  
Let $\delta \in (0, \delta_0)$ be a point of continuity of the nondecreasing function $r \mapsto m(\Aset_r)$.
Then
\begin{equation*}
\liminf\limits_{t \to \infty} \frac{\CGA(\tpm)}{2\epsilon e^{ht}/t} \ge \frac{m(\Aed)}{\Ced}.
\end{equation*}
\end{corollary}

\begin{proof}
By \corref{ConjGU via aaltlmateU} and \lemref{lataU and aaltlataU are close in norm},
\begin{equation*}
\liminf_{t \to \infty} \frac{\CGA(\tpm)}{2\epsilon e^{ht}/t} \ge \frac{1}{\Ced} \cdot \frac{m(\Aed)}{\lmatteinset{\Aed}(\Aed)}.
\end{equation*}
The fact that $\lmatteinset{\Aed}$ is a probability measure gives us the desired inequality.
\end{proof}

Combining \lemref{early upper bound for ConjGA} and \corref{early lower bound for ConjGA}, we obtain the following result.

\begin{proposition}				\label{early bounds for ConjGA}
Let $X$ be a proper $\CAT(0)$ space.
Assume $\G$ acts freely, properly discontinuously, and by isometries on $X$, and that $m_\G$ is finite and mixing.
Let $v_0 \in SX$ be zero-width, and let $\epsilon \in (0, \epsilon_0]$.  
Let $\delta \in (0, \delta_0)$ be a point of continuity of the nondecreasing function $r \mapsto m(\Aset_r)$.
Then for every $\alpha > 0$ there exists $t_0 > 0$ such that for all $t \ge t_0$,
\begin{equation*}
\frac{1 - \alpha}{\Ced} \cdot \frac{2 \epsilon m(\Aed) e^{ht}}{t}
\le
\CGA (\tpm)
\le
(1 + \alpha) \Ced \cdot {2 e^{ht} m(\Aed)}.
\end{equation*}
\end{proposition}

\begin{lemma}					\label{early lower bound for Conj}
Let $U \subseteq SX$ contain a nonempty open set, and let $\alpha > 0$.
There exist $C > 0$ and $t_0 > 0$ such that for all $t \ge t_0$ and 
\[\CGinset{U} (\tpma) \ge C \frac{e^{ht}}{t}.\]
\end{lemma}

\begin{proof}
Let $V \subseteq U$ be a nonempty open set.
By \propref{zero-width geodesics are dense}, there is some zero-width $v_0 \in V$.
By \lemref{Ballmann's lemma}, there exist $\delta > 0$ and $\epsilon > 0$ such that $\Aed = \Aset(v_0,\epsilon,\delta)$ is completely contained in $\Reg \cap V$.
We may assume $\epsilon \le \min \set{\alpha, \epsilon_0}$ and that $\delta \in (0, \delta_0]$ is chosen such that
\begin{equation*}
\liminf\limits_{t \to \infty} \frac{\CGA(\tpm)}{2\epsilon e^{ht}/t} \ge \frac{m(\Aed)}{\Ced}
\end{equation*}
by \corref{early lower bound for ConjGA}.
Thus there exist $C > 0$ and $t_0 > 0$ such that
for all $t \ge t_0$,
\[C \frac{e^{ht}}{t}
\le \CGA (\tpm)
\le \CGinset{U} (\tpm)
\le \CGinset{U} (\tpma).\qedhere\]
\end{proof}

It is easy to see that \lemref{early lower bound for Conj} is equivalent to the following statement, where we replace $\CGinset{U} (\tpma)$ by $\CGinset{U} (\toma)$.

\begin{corollary}				\label{early lower bound for Conj asymmetric}
Let $U \subseteq SX$ contain a nonempty open set, and let $\alpha > 0$.
Then there exist $C > 0$ and $t_0 > 0$ such that for all $t \ge t_0$ and 
\[\CGinset{U} (\toma) \ge C \frac{e^{ht}}{t}.\]
\end{corollary}

\section{Counting Multiplicities}

We start with a simple upper bound on the number of conjugacy classes, coming from the construction of the Patterson--Sullivan measures.

\begin{lemma}					\label{Patterson upper bound}
If $K \subset SX$ is compact, then $\lim\limits_{t \to \infty} e^{-h't} \CGinset{K} (0,t) = 0$ for all $h' > h$.
\end{lemma}

\begin{proof}
Consider that for $\g \in \G$ with an axis in $K$, we know $d(\g p, p) \le \length{\g} + 2 \diam \pi(K)$, and therefore for all $h' > h$,
\begin{align*}
\sum_{t > 0} e^{-h't} \CGinset{K}(t,t)
&= \sum_{\substack{[\g] \in \Conj(\G) \\ \text{with an axis in }K}} e^{-h' \length{\g}}
\le \sum_{\substack{\g \in \G \\ \text{with an axis in }K}} e^{-h' \length{\g}} \\
&\le \sum_{\substack{\g \in \G \\ \text{with an axis in }K}} e^{-h' d(\g p, p) + 2 h' \diam \pi(K)} \\
&\le  e^{2 h' \diam \pi(K)} \sum_{\substack{\g \in \G \\ \text{with an axis in }K}} e^{-h' d(\g p, p)}
\end{align*}
converges because $h$ is the critical exponent of the Poincar\'e series for Patterson's construction.
It follows that $\lim_{t \to \infty} e^{-h't} \CGinset{K} (0,t) = 0$.
\end{proof}

\begin{lemma}					\label{most are prime}
Let $U \subseteq SX$ contain a nonempty open set, and assume $U \subseteq \G K$ for some compact set $K \subseteq SX$.
Then for every $\alpha > 0$,
\[\lim\limits_{t \to \infty} \frac{\CGpinset{U} (\toma)}{\CGinset{U} (\toma)} = 1.\]
\end{lemma}

\begin{remark}
In particular, if $\G$ acts cocompactly on $X$, then $\lim\limits_{t \to \infty} \frac{\CGp (\toma)}{\CG (\toma)} = 1$.
\end{remark}

\begin{proof}
By \corref{early lower bound for Conj asymmetric} and \lemref{Patterson upper bound}, there exist $C \ge 1$ and $t_0 > 0$ such that
\[\CGinset{U} (0,\frac{t}{2}) \le C \cdot \frac{e^{\frac{3}{2} ht}}{2t}
\qquad \text{and} \qquad
\CGinset{U} (\toma) \ge \frac{1}{C} \cdot \frac{e^{2ht}}{2t}\]
for all $t \ge t_0$,
Since every primitive $[\g] \in \ConjGA (\toma) \setminus \ConjGpA (\toma)$ is a multiple of some $[\phi] \in \ConjGA (0,\frac{t}{2})$, we see that
\[\CGinset{U} (\toma) - \CGpinset{U} (\toma) \le \CGinset{U} (0,\frac{t}{2})\]
and therefore
\[1 \ge \frac{\CGpinset{U} (\toma)}{\CGinset{U} (\toma)} \ge 1 - \frac{\CGinset{U} (0,\frac{t}{2})}{\CGinset{U} (\toma)} \ge 1 - e^{-\frac{1}{2} ht}.
\qedhere\]
\end{proof}

Since $\CGinset{U} (0,t)$ diverges, we obtain the following corollary.

\begin{corollary}				\label{prime corollary}
Under the hypotheses of \lemref{most are prime},
\[\lim_{t \to \infty} \frac{\CGpinset{U} (0, t)}{\CGinset{U} (0, t)} = 1.\]
\end{corollary}

It follows from \lemref{most are prime} that the probability measures $\lpataA$ and $\lmataA$ have the same weak limits.
In fact, we have the following.

\begin{lemma}					\label{convergence in norm}
Let $U \subseteq SX$ contain a nonempty open set, and assume $U \subseteq \G K$ for some compact set $K \subseteq SX$.
For any fixed $\alpha > 0$ and choice of axis $\axis$,
\begin{equation*}
\lim\limits_{t \to \infty} \norm{\lpatainset{U} - \lmatainset{U}} = 0.
\end{equation*}
\end{lemma}

\begin{proof}
Let $W$ be a Borel subset of $SX$.
By the definitions,
\begin{align*}
\lmatainset{SX}(W) &= \frac{1}{\CGinset{U} (\toma)} \sum_{[\g] \in \ConjG (\toma)} \frac{1}{\smallnorm{\laG}} \laG(W) \\
\text{and} \qquad
\lpatainset{SX}(W) &= \frac{1}{\CGpinset{U} (\toma)} \sum_{[\g] \in \ConjGp (\toma)} \frac{1}{\smallnorm{\laG}} \laG(W).
\end{align*}
The outside coefficients are asymptotically equal (and nonzero), and the difference in the sums is at most $\CGinset{U} (t - \alpha, t) - \CGpinset{U} (t - \alpha, t)$, which is asymptotically zero compared to $\CGinset{U} (t - \alpha, t)$ by \lemref{most are prime}.
\end{proof}

\section{Limiting Process}
\label{standard limiting section}

For a fixed interval $[a,b] \subset \R$ and continuous function $f \colon [a,b] \to \R$, the Riemann sums ${\sum_{k = 1}^{n} 2\epsilon_n f(x_n)}$ converge to ${\int_{a}^{b} f(x) \; dx}$, for $\epsilon_n = \frac{b-a}{2n}$ and $x_n = (2k-1) \epsilon_n$.
This also holds whenever $f$ is Riemann integrable, e.g.~ $f$ is bounded and nondecreasing.
For completeness, we give here a proof of a standard generalization of this fact to asymptotic intervals.

\begin{lemma}					\label{standard limiting process}
Let $F \colon \R \to \R$ be eventually positive and nondecreasing.
Then
\begin{equation*}
\frac{1}{C}
\le \fakelim_{t \to \infty} \frac%
{\int_{0}^{t} F(x) \; dx}%
{\sum_{k = 0}^{\floor{\frac{t}{2\epsilon}}} 2\epsilon F(t - (2k+1) \epsilon)}
\le C,
\end{equation*}
where $C = \limsup\limits_{x \to \infty} \frac{F(x + \epsilon)}{F(x)}$.
\end{lemma}

\begin{proof}
For any fixed $a \in \R$ and $m \in \Z$,
\begin{equation*}
\lim\limits_{t \to \infty}
\frac{\int_{0}^{t} F(x) \; dx}{\int_{a}^{t} F(x) \; dx}
= 1
\qquad \text{and} \qquad
\lim\limits_{t \to \infty}
\frac{\sum_{k = 0}^{\floor{\frac{t}{2\epsilon}}} 2\epsilon F(t - (2k+1) \epsilon)}%
{\sum_{k = 0}^{\floor{\frac{t}{2\epsilon}} - m} 2\epsilon F(t - (2k+1) \epsilon)}
= 1,
\end{equation*}
so without loss of generality we may assume $F$ is positive and nondecreasing on $[0, \infty)$.
We may similarly assume, for $\alpha > 0$ fixed, that 
$1 \le \frac{F(x + \epsilon)}{F(x)} \le C + \alpha$
for all $x > -2\epsilon$.
Let $t > 0$ and put $n = \floor{\frac{t}{2\epsilon}}$.
For each $k = 0, 1, 2, \dotsc, n$, we have
\begin{equation*}
\frac{1}{C + \alpha} F(t - (2k+1)\epsilon)
\le F(x)
\le (C + \alpha) F(t - (2k+1)\epsilon)
\end{equation*}
for all $x \in [t - (2k+2)\epsilon, t - 2k\epsilon]$.
Thus
\begin{equation*}
\frac{1}{C + \alpha} 2\epsilon F(t - (2k+1)\epsilon)
\le \int_{t - (2k+2)\epsilon}^{t - 2k\epsilon} F(x) \; dx
\le (C + \alpha) 2\epsilon F(t - (2k+1)\epsilon)
\end{equation*}
for each $k = 0, 1, 2, \dotsc, n$, and therefore
\begin{align*}
\frac{1}{C + \alpha} \sum_{k = 0}^{n-1} 2\epsilon F(t - (2k+1)\epsilon)
&\le \int_{2\epsilon}^{t} F(x) \; dx
\le \int_{0}^{t} F(x) \; dx \\
&\le \int_{-2\epsilon}^{t} F(x) \; dx
\le (C + \alpha) \sum_{k = 0}^{n} 2\epsilon F(t - (2k+1)\epsilon).
\end{align*}
But
\begin{equation*}
\lim_{t \to \infty} \frac%
{\sum_{k = 0}^{\floor{\frac{t}{2\epsilon}}} 2\epsilon F(t - (2k+1)\epsilon)}%
{\sum_{k = 0}^{\floor{\frac{t}{2\epsilon}}-1} 2\epsilon F(t - (2k+1)\epsilon)}
= 1,
\end{equation*}
so
\begin{align*}
\frac{1}{C + \alpha}
&\le \fakelim_{t \to \infty} \frac%
{\int_{0}^{t} F(x) \; dx}%
{\sum_{k = 0}^{\floor{\frac{t}{2\epsilon}}} 2\epsilon F(t - (2k+1)\epsilon)}
\le C + \alpha.
\end{align*}
As $\alpha > 0$ was arbitrary, we find
\begin{equation*}
\frac{1}{C}
\le \fakelim_{t \to \infty} \frac{\int_{0}^{t} F(x) \; dx}{\sum_{k = 0}^{n} 2\epsilon F(t - (2k+1)\epsilon)}
\le C.
\qedhere
\end{equation*}
\end{proof}

The following is another standard calculation which we include for completeness.

\begin{lemma}					\label{li(x)}
\begin{equation*}
\frac{1}{C}
\le \fakelim_{t \to \infty} \biggslashfrac%
{\sum_{k = 0}^{\floor{\frac{t}{2\epsilon}}} 2\epsilon \cfrac{e^{h(t - (2k+1) \epsilon)}}{t - (2k+1) \epsilon}}%
{\frac{e^{ht}}{ht}}
\le C,
\end{equation*}
where $C = e^{h \epsilon}$.
\end{lemma}

\begin{proof}
It is a standard fact that for any fixed $t_0 > 0$,
\begin{equation}				\label{eq li(x)}
\lim\limits_{t \to \infty} \biggslashfrac%
{\int_{t_0}^{t} \frac{e^{hx}}{x} \; dx}%
{\frac{e^{ht}}{ht}} = 1.
\end{equation}
This comes from the calculation
\begin{equation*}
\int_{t_0}^{t} \frac{e^{hx}}{x} \; dx
= \frac{e^{hx}}{hx} \bigg \vert_{t_0}^{t} + \int_{t_0}^{t} \frac{e^{hx}}{hx^2} \; dx
= \frac{e^{ht}}{ht} - \frac{e^{h t_0}}{h t_0} + \int_{t_0}^{t} \frac{e^{hx}}{hx^2} \; dx;
\end{equation*}
the second term of the last expression tends to zero relative to $\slashfrac{e^{ht}}{ht}$ because it is constant, the third because $\lim\limits_{x \to \infty} \biggslashfrac{\dfrac{e^{hx}}{hx^2}}{\dfrac{e^{hx}}{x}} = 0$.
On the other hand, for all $\epsilon > 0$,
\lemref{standard limiting process} gives us
\begin{equation*}
e^{-h \epsilon}
\le \fakelim_{t \to \infty} \biggslashfrac%
{\sum_{k = 0}^{\floor{\frac{t}{2\epsilon}}} 2\epsilon \frac{e^{h(t - (2k+1) \epsilon)}}{t - (2k+1) \epsilon}}%
{\int_{t_0}^{t} \frac{e^{hx}}{x} \; dx}
\le e^{h \epsilon}
\end{equation*}
and therefore
\begin{equation*}
e^{-h \epsilon}
\le \fakelim_{t \to \infty} \biggslashfrac%
{\sum_{k = 0}^{\floor{\frac{t}{2\epsilon}}} 2\epsilon \frac{e^{h(t - (2k+1) \epsilon)}}{t - (2k+1) \epsilon}}%
{\frac{e^{ht}}{ht}}
\le e^{h \epsilon}
\end{equation*}
from \eqref{eq li(x)}.
\end{proof}

\section{Entropy and Equidistribution}

Knieper also proves an equidistribution result \cite[Proposition 6.4]{knieper98}; adapting his proof we obtain a similar result.
For clarity, we include a proof.

A significant portion of Knieper's proof of his Proposition 6.4 is spent proving the following (unstated) general lemma.

\begin{lemma}					\label{Knieper entropy variation}
Let $\phi$ be a measurable map of a measurable space to itself.
Let $(\mu_k)$ be a sequence of $\phi$-invariant probability measures, and let $\A$ be a measurable partition.
Then
\[\limsup_{k \to \infty} \frac{H_{\mu_k} (\A_\phi^{(n_k)})}{n_k}
\le \liminf_{k \to \infty} \frac{H_{\mu_k} (\A_\phi^{(q)})}{q}\]
for all integers $q > 1$ and sequences $(n_k)$ in $\N$ such that $n_k \to \infty$.
\end{lemma}

Write $\injrad(\modG{X})$ for the injectivity radius of $\modG{X}$.

\begin{lemma}					\label{Knieper separated}
Let $\G$ be a group acting freely geometrically on a proper, geodesically complete $\CAT(0)$ space $X$ with rank one axis.
Let $t_0 > 0$ and let $P \subset \ConjG(t_0 - \alpha, t_0)$.
If $\alpha < \frac{2}{3} \injrad(\modG{X})$ then $\pr(\axis(P))$ is $(\ceil{t_0},\alpha)$-separated for any choice of axis $\axis$.
\end{lemma}

\begin{proof}
Let $\axis$ be a choice of axis, and let $0 < \alpha < \frac{2}{3} \injrad(\modG{X})$.
Let $\g_1, \g_2 \in \G$ represent distinct conjugacy classes $[\g_1], [\g_2] \in P$.
Let $v = \axis[\g_1]$ and $w = \axis[\g_2]$, and write $\bar v = \pr v$ and $\bar w = \pr w$.
We may assume, replacing $w$ by $\g w$ and $\g_2$ by $\g \g_2 \g^{-1}$ (for some $\g \in \G$) if necessary, that $d(\bar v, \bar w) = d(v,w)$.

Suppose, by way of contradiction, $d(g_\G^n \bar v, g_\G^n \bar w) \le \alpha$ for all $n = 0, 1, 2, \dotsc, \ceil{t_0}$.
Since $\alpha < \injrad(\modG{X})$, we find $d(v(n), w(n)) = d(\bar v(n), \bar w(n)) \le \alpha$ for all such $n$.
Thus $d(v(t), w(t)) \le \alpha$ for all $t \in [0,t_0]$ by convexity.
Find $t_1, t_2 \in [t_0 - \alpha, t_0]$ such that $\g_1 v = g^{t_1} v$ and $\g_2 w = g^{t_2} w$.
Then
\[d(\g_2^{-1} \g_1 v(0), w(0))
= d(\g_2^{-1} v(t_1), \g_2^{-1} w(t_2))
= d(v(t_1), w(t_2))
\le 2\alpha.\]
Hence $d(\g_2^{-1} \g_1 v(0), v(0)) \le 3\alpha < 2\injrad(\modG{X})$, which is only possible if $\g_2^{-1} \g_1$ is trivial.
This contradicts our hypothesis that $[\g_1]$ and $[\g_2]$ are distinct.
Therefore, there must be some $n \in \set{0, 1, 2, \dotsc, \ceil{t_0}}$ such that $d(g_\G^n \bar v, g_\G^n \bar w) > \alpha$, and thus we see that $\pr(\axis(P))$ is $(\ceil{t_0},\alpha)$-separated.
\end{proof}

\begin{definition}
Let $P \subset \ConjG$ be finite.
Call a $g^t$-invariant probability measure $\nu$ on $\modG{SX}$ \defn{equal-weighted} along $\axis(P)$ if $\nu$ gives measure $\frac{1}{\card{P}}$ to the orbit of $\pr(\axis[\g])$ for each $[\g] \in P$, where $\pr \colon SX \to \modG{SX}$ is the canonical projection map.
\end{definition}

\begin{proposition}				\label{Knieper equidistribution}
Let $\G$ be a group acting freely geometrically on a proper, geodesically complete $\CAT(0)$ space $X$ with rank one axis.
Let $(\nu_k)$ be a sequence of $g^t$-invariant probability measures on $\modG{SX}$.
Assume there exists $\epsilon_0$ such that $0 < \epsilon_0 < \frac{2}{3} \injrad(\modG{X})$ and each $\nu_k$ is equal-weighted along $\axis(P_k)$ for some choice of axis $\axis$ and subset $P_k \subset \ConjG(t_k - \epsilon_0, t_k)$, where $t_k \to \infty$ as $k \to \infty$.
If
\[\lim_{k \to \infty} \frac{\log \card{P_k}}{t_k} = h\]
then $\nu_k \to m_\G$ weakly.
\end{proposition}

\begin{proof}
By compactness of the space of $g^t$-invariant Borel probability measures on $\modG{SX}$ under the weak* topology, every subsequence $(\nu_{k_j})$ has at least one weak* accumulation point $\nu$ of $\set{\nu_k}$.
By uniqueness of the measure of maximal entropy, it suffice to prove that every such $\nu$ is a measure of maximal entropy for $g_\G^t$.

Let $\nu$ be a weak* accumulation point of $\set{\nu_k}$; passing to a subsequence if necessary, we may assume $\nu_k \to \nu$ in the weak* topology.
Fix a measurable partition $\A = \set{A_1, \dotsc, A_m}$ of $\modG{SX}$ such that $\diam \A \le \delta < \epsilon_0$ and $\nu(\bd A_i) = 0$.
Let $n_k = \ceil{t_k}$.
Since the closed geodesics in $\pr(\axis(P_k))$ are $(n_k,\delta)$-separated by \lemref{Knieper separated}, each $\varalpha \in \A_\phi^{(n_k)}$ touches at most one orbit from $\pr(\axis(P_k))$, and thus $\nu_k(\varalpha) \le \frac{1}{\card{P_k}}$.
Therefore the entropy
\begin{align*}
H_{\nu_k}(\A_\phi^{(n_k)})
= - \sum_{\varalpha \in \A_\phi^{(n_k)}} \nu_k(\varalpha) \log \nu_k(\varalpha)
\ge \sum_{\varalpha \in \A_\phi^{(n_k)}} \nu_k(\varalpha) \log \card{P_k}
= \log \card{P_k}.
\end{align*}
Since $\nu(\bd A_i) = 0$ for all $A_i \in \A$, we have
$H_{\nu_k}(\A_\phi^{(q)}) \to H_{\nu}(\A_\phi^{(q)})$
and thus
\begin{align*}
h_{\nu}(\phi)
\ge h_{\nu}(\phi, \A)
= \lim_{q \to \infty} \frac{H_{\nu} (\A_\phi^{(q)})}{q}
= \lim_{q \to \infty} \lim_{k \to \infty} \frac{H_{\nu_k} (\A_\phi^{(q)})}{q}.
\end{align*}
By \lemref{Knieper entropy variation} and the inequality
$H_{\nu_k}(\A_\phi^{(n_k)}) \ge \log \card{P_k}$
from above,
\begin{align*}
\lim_{q \to \infty} \lim_{k \to \infty} \frac{H_{\nu_k} (\A_\phi^{(q)})}{q}
\ge \fakelim_{k \to \infty} \frac{H_{\nu_k} (\A_\phi^{(n_k)})}{n_k}
\ge \lim_{k \to \infty} \frac{\log \card{P_k}}{t_k}
= h.
\end{align*}
Therefore $h_{\nu}(\phi) \ge h$, which shows that $\nu$ is a measure of maximal entropy.
\end{proof}

\begin{corollary}				\label{slow singular growth 2}
Let $\G$ be a group acting freely geometrically on a proper, geodesically complete $\CAT(0)$ space $X$ with rank one axis.
Then
\begin{equation*}
\limsup_{t \to \infty} \frac{\log \CGinset{SX \setminus \Reg}(t-\epsilon_0,t)}{t}
= 0
\end{equation*}
for all $\epsilon_0 \in (0, \frac{2}{3} \injrad(\modG{SX}))$.
In particular,
\begin{equation*}
\limsup_{t \to \infty} \frac{\log \CGinset{SX \setminus \Reg}(0,t)}{t}
= 0.
\end{equation*}
\end{corollary}

\begin{proof}
Suppose not.
Then we have $\epsilon_0 \in (0, \frac{2}{3} \injrad(\modG{SX}))$ and $t_k \to \infty$ such that the sets $P_k = \ConjGpinset{SX \setminus \Reg}(t_k - \epsilon_0, t)$ satisfy $\lim_{k \to \infty} \frac{\log \card{P_k}}{t_k} = h$.
Hence by \propref{Knieper equidistribution}, $\lambda_{\axis,t_k,\alpha}^{{\rm prime}, SX \setminus \Reg} \to m_\G$ weakly.
But $SX \setminus \Reg$ is closed in $SX$, so $m_\G$ must be supported on $\modG{(SX \setminus \Reg)}$, which contradicts the fact that $m_\G$ is supported on $\Reg$.
Therefore, the statement of the corollary must hold.
\end{proof}

\begin{theorem}					\label{equidistribution on U}
Let $\G$ be a group acting freely geometrically on a proper, geodesically complete $\CAT(0)$ space $X$ with rank one axis.
Let $U \subseteq SX$ contain a nonempty open set.
For any fixed $\alpha$ with $0 < \alpha < \frac{2}{3} \injrad(\modG{X})$ and choice of axis $\axis$, the measures $\lpatainset{U}$ converge weakly to $m_\G$.
\end{theorem}

\begin{proof}
Let $(t_k)$ be a sequence of positive real numbers such that $t_k \to \infty$.
Let $P_k = \ConjG(t_k - \alpha, t_k)$.
By \corref{early lower bound for Conj asymmetric},
$\lim_{k \to \infty} \frac{\log \card{P_k}}{t_k} = h$,
and thus $\lambda_{\axis,t_k,\alpha}^{{\rm prime}, U} \to m_\G$ weakly by \propref{Knieper equidistribution}.
Since $(t_k)$ was arbitrary, it follows that the measures $\lpatainset{U}$ converge weakly to $m_\G$.
\end{proof}

\section{Using Equidistribution}

From \thmref{equidistribution on U} we obtain the following.

\begin{lemma}					\label{alternate equidistribution on A}
Let $\G$ be a group acting freely geometrically on a proper, geodesically complete $\CAT(0)$ space $X$ with rank one axis.
Let $U \subseteq SX$ contain a nonempty open set.
For any fixed $\alpha$ with $0 < \alpha < \frac{1}{3} \injrad(\modG{X})$ and choice of axis $\axis$, the measures $\aaltlmatainset{U}$ converge weakly to $m_\G$.
\end{lemma}

\begin{proof}
By \thmref{equidistribution on U}, the measures $\lpattainset{U}$ converge weakly to $m_\G$.
Then by \lemref{convergence in norm} and \lemref{lataU and aaltlataU are close in norm}, the measures $\lmattainset{U}$ and $\aaltlmatainset{U}$ do likewise.
\end{proof}

\begin{lemma}					\label{ConjGU asymptotics}
Let $\G$ be a group acting freely geometrically on a proper, geodesically complete $\CAT(0)$ space $X$ with rank one axis.
Fix a zero-width geodesic $v_0 \in SX$.
Let $\epsilon \in (0, \epsilon_0]$ and $\delta \in (0, \delta_0)$ be small enough that $\epsilon < \frac{1}{3} \injrad(\modG{X})$.
Assume $\delta \in (0, \delta_0)$ is a point of continuity of the nondecreasing function $r \mapsto m(\Aset_r)$.
Then
\begin{equation*}
\frac{1}{\Ced} \le \fakelim_{t \to \infty} \frac{\CGinset{U}(\tpm)}{2\epsilon e^{ht}/t} \le \Ced
\end{equation*}
for all $U$ satisfying $\Aed \subseteq U \subseteq SX$.
\end{lemma}

\begin{proof}
Since $\Aed$ is a continuity set for $m$ and $\diam \Aed < \injrad(\modG{SX})$, by \lemref{alternate equidistribution on A} we see that
\(\lim_{t \to \infty} \aaltlmatainset{U} (\pr \Aed) = m_\G(\pr \Aed) = m(\Aed)\).
Apply \corref{ConjGU via aaltlmateU}.
\end{proof}

Putting $F(t) = e^{ht} / t$ in \lemref{standard limiting process},
by \lemref{ConjGU asymptotics} we obtain our desired asymptotics for $\CGinset{U}(0,t)$.
But to do so, we need to check the overlaps we get from counting the endpoints of closed intervals are asymptotically small.

\begin{lemma}					\label{ConjGU asymptotics from zero}
Let $\G$ be a group acting freely geometrically on a proper, geodesically complete $\CAT(0)$ space $X$ with rank one axis.
Fix a zero-width geodesic $v_0 \in SX$.
Let $\epsilon, \delta > 0$ be small enough that $\diam \Aed < \injrad(\modG{SX})$ and $\epsilon < \frac{1}{3} \injrad(\modG{X})$.
Assume $\delta$ is chosen so that $\Aed$ is a continuity set for $m$.
Then
\begin{equation*}
\frac{1}{e^{h \epsilon} \Ced} \le \fakelim_{t \to \infty} \frac{\CGinset{U}(0,t)}{e^{ht}/ht} \le e^{h \epsilon} \Ced
\end{equation*}
for all $U \subseteq SX$ such that $\Aed \subseteq U \subseteq SX$.
\end{lemma}

\begin{proof}
By \lemref{ConjGU asymptotics}, for all $\alpha \in (0, \epsilon]$ and $U$ such that $\Aed \subseteq U \subseteq SX$, we have
\begin{equation*}
\frac{1}{\Cad} \le \fakelim_{t \to \infty} \frac{\CGinset{U}(\tpma)}{2\alpha e^{ht}/t} \le \Cad,
\end{equation*}
and therefore
\begin{equation*}
\frac{1}{\Cad} \le \fakelim_{t \to \infty} \cfrac%
{\sum_{k = 0}^{\floor{\frac{t}{2\epsilon}}} {\CGinset{U} \big( t - (2k+1)\epsilon - \alpha, t - (2k+1)\epsilon + \alpha \big)}}%
{\sum_{k = 0}^{\floor{\frac{t}{2\epsilon}}} 2\alpha \cfrac{e^{h(t - (2k+1) \epsilon)}}{t - (2k+1) \epsilon}}
\le \Cad.
\end{equation*}
Since for all $\alpha \in (0, \epsilon)$,
\begin{align*}
&\sum_{k = 0}^{\floor{\frac{t}{2\epsilon}}} {\CGinset{U} \big( t - (2k+1)\epsilon - \alpha, t - (2k+1)\epsilon + \alpha \big)} \\
&\le
\CGinset{U}(0,t)
\le
\sum_{k = 0}^{\floor{\frac{t}{2\epsilon}}} {\CGinset{U} \big( t - (2k+2)\epsilon, t - 2k\epsilon \big)},
\end{align*}
letting $\alpha \to \epsilon$ from below gives us
\begin{equation*}
\frac{1}{\Ced}
= \lim_{\alpha \to \epsilon^-} \frac{1}{\Cad}
\le \fakelim_{t \to \infty} \cfrac%
{\CGinset{U}(0,t)}%
{\sum_{k = 0}^{\floor{\frac{t}{2\epsilon}}} 2\epsilon \cfrac{e^{h(t - (2k+1) \epsilon)}}{t - (2k+1) \epsilon}}
\le \Ced.
\end{equation*}
Thus by \lemref{li(x)},
\begin{equation*}
\frac{1}{C} \cdot \frac{1}{\Ced} \le \fakelim_{t \to \infty} \frac{\CGinset{U}(0,t)}{e^{ht}/ht} \le C \cdot \Ced,
\end{equation*}
where $C = e^{h \epsilon}$.
\end{proof}

\begin{theorem}					\label{main counting}
Let $\G$ be a group acting freely geometrically on a proper, geodesically complete $\CAT(0)$ space $X$ with rank one axis.
Let $U \subseteq SX$ contain a nonempty open set.
Then
\begin{equation*}
\lim_{t \to \infty} \frac{\CGpinset{U}(0,t)}{e^{ht}/ht} = \lim_{t \to \infty} \frac{\CGinset{U}(0,t)}{e^{ht}/ht} = 1.
\end{equation*}
\end{theorem}

\begin{proof}
Let $v_0$ be a zero-width geodesic in the interior of $U$.
Then
$\lim_{\epsilon,\delta \to 0} \Ced = 1$, so the second equality holds by \lemref{ConjGU asymptotics from zero}.
The first equality holds by \corref{prime corollary}.
\end{proof}

In particular, this holds for $U = \Reg$ and $U = SX$:
\begin{align*}
&\lim_{t \to \infty} \frac{\CGpinset{\Reg}(0,t)}{e^{ht}/ht} = \lim_{t \to \infty} \frac{\CGinset{\Reg}(0,t)}{e^{ht}/ht} \\
= &\lim_{t \to \infty} \frac{\CGp(0,t)}{e^{ht}/ht} = \lim_{t \to \infty} \frac{\CG(0,t)}{e^{ht}/ht} = 1.
\end{align*}
This proves \thmref{main}.

\section{Finish}

Much of the proof of \thmref{main counting} goes through without assuming cocompactness.
In particular, what we used was equidistribution (the conclusion of \lemref{alternate equidistribution on A}) for the second equality and \corref{prime corollary} for the first.

\bibliographystyle{amsplain}
\bibliography{refs}

\end{document}

%% file: artcfg.tex
\usepackage{amssymb}
\usepackage{hyperref}
\usepackage{tikz}

\usepackage{environ}

\usepackage{import}

\theoremstyle{plain}
\newtheorem{theorem}[equation]{Theorem}
\newtheorem*{theorem*}{Theorem}

\newtheorem{corollary}[equation]{Corollary}
\newtheorem{lemma}[equation]{Lemma}
\newtheorem{conj}[equation]{Conjecture}

\newtheorem{proposition}[equation]{Proposition}
\newtheorem{fact}[equation]{Fact}
\newtheorem*{fact*}{Fact}
\newtheorem*{question*}{Question}

\newtheorem{main}{Theorem}

\newtheorem{maintheorem}[main]{Theorem}

\theoremstyle{remark}

\newtheorem*{remark}{Remark}

\theoremstyle{definition}
\newtheorem{definition}[equation]{Definition}

\newtheorem*{standing hypothesis}{Standing Hypothesis}

\newcommand{\thmref}[1]{Theorem~\ref{#1}}
\newcommand{\corref}[1]{Corollary~\ref{#1}}
\newcommand{\lemref}[1]{Lemma~\ref{#1}}
\newcommand{\propref}[1]{Proposition~\ref{#1}}

\newcommand{\itemrefstar}[1]{(\ref*{#1})}

\newcommand{\secref}[1]{Section~\ref{#1}}

\newcommand{\defn}[1]{\emph{#1}}

\graphicspath{{pictures/}}

\newcommand{\N}{\mathbb{N}}
\newcommand{\Z}{\mathbb{Z}}

\newcommand{\R}{\mathbb{R}}

\renewcommand{\H}{\mathbb{H}}

\renewcommand{\setminus}{\smallsetminus}

\DeclareMathOperator{\injrad}{injrad}
\DeclareMathOperator{\diam}{diam}
\DeclareMathOperator{\id}{id}
\DeclareMathOperator{\Isom}{Isom}

\newcommand{\slashfrac}[2]{{#1} / {#2}}

\newcommand{\biggslashfrac}[2]{{#1} \bigg/ {#2}}

\newcommand{\set}[1]{\left\{#1\right\}}
\newcommand{\setp}[2]{\left\{#1 : #2\right\}}
\newcommand{\family}[1]{\left(#1\right)}
\newcommand{\abs}[1]{\left| #1 \right|}
\newcommand{\norm}[1]{\left\| #1 \right\|}
\newcommand{\smallnorm}[1]{\| #1 \|}

\DeclareMathOperator*{\fakelim}{\widetilde \lim}

\newcommand{\cl}[1]{\overline{#1}}
\newcommand{\bd}{\partial}
\newcommand{\double}[1]{#1 \times #1}
\newcommand{\dbX}{\double{\bd X}}

\newcommand{\Reg}{\mathcal R}
\newcommand{\RE}{\Reg^E}

\newcommand{\Zerowidth}{\mathcal{Z}}

\newcommand{\lmod}{\backslash}

\newcommand{\modgp}[2]{#2 \lmod #1}
\newcommand{\modG}[1]{\modgp{#1}{\Gamma}}

\DeclareMathOperator{\pr}{pr}

\DeclareMathOperator{\emap}{E}

\DeclareMathOperator{\A}{\mathcal A}

\newcommand{\G}{\Gamma}

\newcommand{\g}{\gamma}

\DeclareMathOperator{\dT}{d_T}
\renewcommand{\epsilon}{\varepsilon}
\DeclareMathOperator{\width}{width}
\DeclareMathOperator{\Limitset}{\Lambda}

\DeclareMathOperator{\BT}{B_T}
\DeclareMathOperator{\Bcone}{B_{\rho}}

\newcommand{\length}[1]{\abs{#1}}
\DeclareMathOperator{\CAT}{CAT}

\newcommand{\ceil}[1]{\lceil #1 \rceil}

\newcommand{\floor}[1]{\lfloor #1 \rfloor}

\newcommand{\card}[1]{\# {#1}}

\newcommand{\Aset}{\ensuremath \mathfrak{v}}
\newcommand{\Bset}{\ensuremath \mathfrak{w}}

\newcommand{\Bu}{\ensuremath \Bset_{\rm unclipped}}
\newcommand{\Iu}{\ensuremath I^{\rm unclipped}}
\newcommand{\Nu}{\ensuremath N^{\rm unclipped}}

\newcommand{\Aed}{\ensuremath \Aset_{\epsilon,\delta}}
\newcommand{\Azd}{\ensuremath \Aset_{0,\delta}}
\newcommand{\Ced}{\ensuremath C_{\epsilon,\delta}}
\newcommand{\Cad}{\ensuremath C_{\alpha,\delta}}
\newcommand{\Cedpn}{\ensuremath C_{\epsilon,\delta'_n}}

\newcommand{\Iedt}{\ensuremath I_{\epsilon,\delta,t}}
\newcommand{\Iperedt}{\ensuremath I_{\epsilon,\delta,t}^{\rm periodic}}
\newcommand{\Nedt}{\ensuremath N_{\epsilon,\delta,t}}
\newcommand{\Gedt}{\ensuremath N_{\epsilon,\delta,t}^{\rm periodic}}

\newcommand{\Iuedt}{\ensuremath I_{\epsilon,\delta,t}^{\rm unclipped}}
\newcommand{\Nuedt}{\ensuremath N_{\epsilon,\delta,t}^{\rm unclipped}}

\newcommand{\tpm}{t - \epsilon, t + \epsilon}

\newcommand{\tpma}{t - \alpha, t + \alpha}
\newcommand{\toma}{t - \alpha, t}

\DeclareMathOperator{\axis}{\mathfrak a}

\DeclareMathOperator{\segclassesA}{\mathfrak S^{\Aed}}
\DeclareMathOperator{\segcount}{S}
\newcommand{\segcountA}{\segcount^{\Aed}}

\newcommand{\Conj}{\mathfrak C}
\newcommand{\Conjp}{\Conj^{\rm prime}}
\newcommand{\Conjinset}[1]{\Conj^{#1}}
\newcommand{\Conjpinset}[1]{\Conj^{{\rm prime}, #1}}

\newcommand{\ConjG}{\Conj_\G}
\newcommand{\ConjGp}{\Conj_\G^{\rm prime}}
\newcommand{\ConjGinset}[1]{\Conj_\G^{#1}}
\newcommand{\ConjGpinset}[1]{\Conj_\G^{{\rm prime}, #1}}

\newcommand{\ConjGA}{\ConjGinset{\Aed}}
\newcommand{\ConjGpA}{\ConjGpinset{\Aed}}

\newcommand{\CG}{{\rm Conj}_\G}
\newcommand{\CGp}{{\rm Conj}_\G^{\rm prime}}
\newcommand{\CGinset}[1]{{\rm Conj}_\G^{#1}}
\newcommand{\CGpinset}[1]{{\rm Conj}_\G^{{\rm prime}, #1}}

\newcommand{\CGA}{\CGinset{\Aed}}

\newcommand{\lv}{\lambda^{v}}
\newcommand{\lvG}{\lambda_{\G}^{v}}
\newcommand{\laG}{\lambda_{\G}^{\axis[\g]}}
\newcommand{\lata}{\lambda_{\axis,t,\alpha}}

\newcommand{\lmatainset}[1]{\lata^{{\rm mult}, #1}}
\newcommand{\lpatainset}[1]{\lata^{{\rm prime}, #1}}

\newcommand{\lmataA}{\lmatainset{\Aed}}
\newcommand{\lpataA}{\lpatainset{\Aed}}

\newcommand{\aaltlata}{\tilde{\lambda}_{\axis,t,\alpha}}

\newcommand{\aaltlmatainset}[1]{\aaltlata^{{\rm mult}, #1}}

\newcommand{\aaltlate}{\tilde{\lambda}_{\axis,t,\epsilon}}
\newcommand{\aaltlmate}{\aaltlate^{\rm mult}}

\newcommand{\aaltlmateinset}[1]{\aaltlate^{{\rm mult}, #1}}

\newcommand{\latta}{\lambda_{\axis, t + \alpha, 2 \alpha}}

\newcommand{\lmattainset}[1]{\latta^{{\rm mult}, #1}}
\newcommand{\lpattainset}[1]{\latta^{{\rm prime}, #1}}
\newcommand{\latte}{\lambda_{\axis, t + \epsilon, 2 \epsilon}}

\newcommand{\lmatteinset}[1]{\latte^{{\rm mult}, #1}}